\documentclass[a4paper,11pt]{article}
\usepackage{pdfsync}

\usepackage[plainpages=false]{hyperref}
\usepackage{amsfonts,amsmath,amssymb,amsthm,mathtools}
\usepackage{latexsym,lscape,rawfonts,mathrsfs,slashed}
\usepackage{enumitem}
\usepackage{scalerel}[2014/03/10]
\usepackage[usestackEOL]{stackengine}
\usepackage{CJKutf8}

\usepackage[all]{xy}
\usepackage{makeidx}
\usepackage{graphicx,psfrag}
\usepackage{pstool}
\usepackage{tikz-cd}

\usepackage{array,tabularx}

\usepackage{setspace}

\usepackage[titletoc,title]{appendix}
\usepackage{txfonts}

\newcommand{\R}{\ensuremath{\mathbb{R}}}

\def\XX{\mathcal{X}}
\def\RR{\mathcal{R}}
\def\MM{\mathcal{M}}
\def\LL{\mathcal{L}}
\def\BB{\mathcal{B}}
\def\Tr{\mathrm{Tr}}
\def\KK{\mathcal{K}}

\def\rE{\mathbf{r}_E}

\def\r{\textbf{r}}
\def\d{\mathrm{d}}

\newcommand{\ba}{\begin{align*}}
	\newcommand{\ea}{\end{align*}}

\newcommand{\na}{\nabla}
\newcommand{\la}{\langle}
\newcommand{\ra}{\rangle}

\newcommand{\lc}{\left(}
\newcommand{\rc}{\right)}

\newcommand{\ep}{\epsilon}

\def\Th{\Theta}

\def\cc{\mathfrak{c}}
\def\scal{\mathrm{R}}

\def\vol{\mathrm{Vol}}

\def\HHH{\mathscr{H}}

\def\CCC{\mathscr{C}}
\def\CC{\mathcal{C}}

\def\NNN{\mathscr{N}}

\def\WW{\mathcal{W}}

\def\Div{\mathrm{div}}

\def\VV{\mathcal{V}}
\def\III{\mathbb{I}}
\def\isd{\mathrm{ISD}}

\newcommand{\Ric}{\ensuremath{\mathrm{Ric}}}

\renewcommand{\t}{\mathfrak{t}}

\def\rBC{\textbf{r}_{B,\sigma}}

\def\rA{\textbf{r}_A}
\def\MS{\mathcal{S}}

\makeatletter
\def\ExtendSymbol#1#2#3#4#5{\ext@arrow 0099{\arrowfill@#1#2#3}{#4}{#5}}

\makeatother

\makeatletter
\def\ExtendSymbol#1#2#3#4#5{\ext@arrow 0099{\arrowfill@#1#2#3}{#4}{#5}}

\makeatother

\DeclarePairedDelimiter\abs{\lvert}{\rvert}%
\makeatletter
\let\oldabs\abs
\def\abs{\@ifstar{\oldabs}{\oldabs*}}
\makeatother

\def\aint{\,\ThisStyle{\ensurestackMath{%
			\stackinset{c}{.2\LMpt}{c}{.5\LMpt}{\SavedStyle-}{\SavedStyle\phantom{\int}}}%
		\setbox0=\hbox{$\SavedStyle\int\,$}\kern-\wd0}\int}

\numberwithin{equation}{section}

\newtheorem{thm}{Theorem}[section]
\newtheorem{cor}[thm]{Corollary}
\newtheorem{prop}[thm]{Proposition}
\newtheorem{lem}[thm]{Lemma}

\newtheorem{defn}[thm]{Definition}

\newtheorem{exmp}[thm]{Example}

\newtheorem{claim}[thm]{Claim}

\title{Strong uniqueness and rectifiability of generalized cylindrical singularities in Ricci flow}
\author{Hanbing Fang \quad and \quad Yu Li} 
\date{\today}

\begin{document}
\maketitle

\begin{abstract}
In this paper, we extend the results of \cite{fang2025strong, fang2025singular} to generalized cylinders. More precisely, we establish a Lojasiewicz inequality for the pointed \(\WW\)-entropy in Ricci flow under the assumption that the geometry near the base point is close to a generalized cylinder \(\mathbb{R}^k \times N^{n-k}\), where \(N\) is an Einstein manifold with obstruction of order three satisfying a suitable spectral condition. As an application, we prove the strong uniqueness of generalized cylindrical tangent flows. Furthermore, we show that the subset $\MS^k_{\mathrm{qc}}(N)\subset \mathcal{S}^k$, consisting of points at which some tangent flow is given by \(\mathbb{R}^k \times N^{n-k}\) or its quotient, is horizontally parabolic \(k\)-rectifiable.
\end{abstract}

	\tableofcontents
    
\section{Introduction}

The analysis of singularities is one of the central themes in Ricci flow. 
Given a Ricci flow developing a finite-time singularity, one studies the local geometry near the singular point by taking parabolic blow-ups. Any such blow-up limit is called a tangent flow. In general, the tangent flow may depend on the choice of rescaling sequence, and proving its uniqueness is a subtle problem.

In this paper, we study the strong uniqueness problem for tangent flows modeled on generalized cylinders
\[
    \mathbb R^k \times N^{n-k},
\]
where \(N^{n-k}\) is a closed Einstein manifold satisfying
\[
    \operatorname{Ric}(g_N)=\frac12 g_N.
\]
This generalizes the cylindrical case considered in \cite{fang2025strong}, where the Einstein factor is the round sphere \(S^{n-k}\). The passage from the round sphere to a general Einstein factor introduces new difficulties. In particular, the kernel of the stability operator is no longer generated only by the quadratic Hermite modes on the Euclidean factor; it also contains infinitesimal solitonic deformations coming from \(N\). Thus the variational analysis of Perelman's \(\WW\)-functional must take into account the deformation theory of the Einstein factor.

More precisely, we consider the generalized weighted cylinder
\[
    \left(\CC_N^k\right)_{-1}:=(\bar M,\bar g,\bar f)
    =
    \left(
        \mathbb R^k \times N^{n-k},
        g_E+g_N,
        \frac{|\vec x|^2}{4}+\frac{n-k}{2}+\Theta_{N,n-k}
    \right),
\]
where \(\Theta_{N,n-k}\) is chosen so that the weighted volume is normalized (see Example \ref{exa:cylinder}). We assume that \(N\) has an obstruction of order \(3\) (see Definition \ref{def:obs}) and satisfies the spectral condition
\[
    -\frac l2 \notin \operatorname{spec}\left(\LL_N|_{\mathrm{TT}}\right),
    \qquad \forall l\in \mathbb N^+,
\]
where $\operatorname{spec}\left(\LL_N|_{\mathrm{TT}}\right)$ denotes the spectrum of $\LL_N$ restricted to the traceless-transverse directions. These assumptions are precisely those used in the rigidity theory for generalized cylindrical Ricci shrinkers developed in \cite{li2023rigidity}. They allow one to control the non-integrable directions coming from the Einstein factor and lead to a weaker, but still effective, Lojasiewicz exponent. Examples of such Einstein manifolds \(N\) include the standard \(\mathbb{CP}^k\) equipped with the Fubini--Study metric (see \cite{kroncke2016rigidity} and \cite{li2023rigidityCPn}), $S^2 \times S^2$ equipped with the standard product metric (see \cite{sunzhu25}), and many compact symmetric spaces (see \cite{caohe15}).

Let $\mathcal{C}_N^k = (\bar M, (\bar g(t))_{t<0}, (\bar f(t))_{t<0})$ denote the corresponding Ricci flow with potential function such that $t=0$ is the singular time. Then $\bar{\mathcal{C}}_N^k$ is defined as the completion of $(\bar M, (\bar g(t))_{t \in (-\infty, 0)})$ with respect to the spacetime distance $d_{\mathcal{C}}^*$ (see \cite[Section 9]{fang2025RFlimit}). The base point $p^*$ is taken to be the limit of $(\bar p, t)$ as $t \nearrow 0$, where $\bar p$ is a minimum point of $\bar f$.

Our first main result is a strong uniqueness theorem for generalized cylindrical tangent flows. Let
\[
   \XX=\{M^n,(g(t))_{t\in[-T,0)}\}
\]
be a closed Ricci flow with entropy bounded below by $-Y$, where \(t=0\) is the first singular time. Let \((Z,d_Z,\mathfrak t)\) be the \(d^*\)-completion of \(X_{[-0.98T,0)}\) (see \cite[Section 9]{fang2025RFlimit}), and let \(z\in Z_0\) be a singular point such that one tangent flow at \(z\) is given by \(\bar{\mathcal{C}}_N^k\). Let \(f_z\) be the potential function based at \(z\), and define the modified Ricci flow
\begin{align} \label{equ:MRFintroI}
       g^z(s)=e^s\varphi_{-e^{-s}}^*g(-e^{-s}),
    \qquad
    f^z(s)=\varphi_{-e^{-s}}^*f_z(-e^{-s}), 
\end{align}
where \(\varphi_t\) is generated by \(-\nabla f_z(t)\) with $\varphi_{-T} = \mathrm{id}$. We prove that this modified flow converges to the generalized cylinder under a fixed gauge, with a quantitative rate. Compared with the round cylindrical case in \cite{fang2025strong}, the convergence rate is weaker, reflecting the larger deformation space of the Einstein factor.

\begin{thm}[Strong uniqueness of generalized cylindrical tangent flows] \label{thm:stong}
Suppose that \((N^{n-k},g_N)\) satisfies the assumptions above. If \(z\in Z_0\) has one tangent flow that is isometric to \(\bar{\mathcal{C}}_N^k\), then the generalized cylindrical tangent flow at \(z\) is strongly unique. More precisely, for every sufficiently small \(\ep>0\), there exists \(\bar s\) such that for every integer \(j\ge \bar s\), there is a diffeomorphism
\[
    \psi_j:\Omega_j:=\{\bar f\le (1-\ep)\log j\}\subset \bar M
    \longrightarrow M
\]
onto its image, satisfying \(\psi_{j+1}=\psi_j\) on \(\Omega_j\), and for all \(s\ge j\),
\[
    \bigl[\psi_j^*g^z(s)-\bar g\bigr]_{[\ep^{-1}]}
    +
    \bigl[\psi_j^*f^z(s)-\bar f\bigr]_{[\ep^{-1}]}
    \le
    C(n,Y,\ep)e^{\bar f/2}s^{-1/2+\ep}.
\]
Here, for any integer $l \ge 0$, the norm $[\cdot]_l$ is defined by
	\begin{align*}
		[\cdot]_l := \sum_{i=0}^l \bigl|\nabla_{\bar g}^i (\cdot) \bigr|_{\bar g}.
	\end{align*}
\end{thm}

The analytic core of the proof is a Lojasiewicz inequality for the pointed \(\WW\)-entropy (see \cite[Definition 2.7]{fang2025strong}). In the round cylindrical case, the leading unstable directions are controlled by a cubic term, which leads to an upper bound of \(3/4\) for the exponent $\gamma$ (see \eqref{equ:lojaintro}). In the present generalized setting, the obstruction-order-three assumption on \(N\) gives a different balance between the Euclidean Hermite modes and the infinitesimal solitonic deformations of \(N\). As a result, we obtain the following Lojasiewicz inequality.

\begin{thm}[Lojasiewicz inequality] \label{thm:loja}
Let \(z\in Z_0\) be a generalized cylindrical singularity modeled on \(\bar{\mathcal{C}}_N^k\). Then for any \(\gamma\in(0,2/3)\), there exists a constant \(C=C(n,N,Y)\) such that, for all sufficiently small \(\tau>0\),
\begin{align}\label{equ:lojaintro}
|\WW_z(\tau)-\Theta_{N,n-k}|
    \le
    C\bigl(\WW_z(\tau/2)-\WW_z(2\tau)\bigr)^\gamma .	
\end{align}

\end{thm}

The proof of this inequality follows the general contraction--extension strategy of \cite{colding2021singularities,li2023rigidity,fang2025strong}, but the variational estimate near the generalized cylinder is substantially different. The main technical step is to expand the \(\WW\)-functional near
\[
    \mathbb R^k\times N^{n-k}
\]
and separate the perturbation into three parts: the quadratic Hermite component on \(\mathbb R^k\), the infinitesimal solitonic deformation component from \(N\), and the remaining coercive part. The obstruction condition on \(N\) is then used to control the Einstein-factor deformation, while the spectral condition rules out additional resonances. This yields a variational inequality of the form
\[
    |\WW(g,f)-\WW(\bar g,\bar f)|
    \le
    C \mathfrak{X}^{4/3},
\]
for normalized compactly supported perturbations, where \(\mathfrak{X}\) measures the Ricci shrinker defect, the gauge error, and the center-of-mass error; see Corollary \ref{cor:lojaforF}.

Finally, we apply the Lojasiewicz inequality to the structure of singular sets in noncollapsed Ricci flow limit spaces. Let \(\MS^k(N)\subset \MS^k\) denote the set of points at which some tangent flow is given by the generalized cylinder \(\bar{\mathcal{C}}_N^k\). Combining the entropy summability obtained here (see Proposition \ref{sumWonRFlimit}) with the cylindrical neck region argument of \cite{fang2025singular} (see Definition \ref{defiofcylneckregion}), we prove that \(\MS^k(N)\) is horizontally parabolic \(k\)-rectifiable. The same argument also applies to finite quotients of generalized cylinders.

\begin{thm}[Rectifiability of generalized cylindrical singular sets] 
The set \(\MS^k(N)\) is horizontally parabolic \(k\)-rectifiable with respect to the spacetime distance \(d_Z\). More generally, the subset consisting of points whose tangent flow is \(\bar{\mathcal{C}}_N^k\) or a finite quotient thereof is horizontally parabolic \(k\)-rectifiable.
\end{thm}

The paper is organized as follows. Section~\ref{sec:pre} recalls the necessary definitions and preliminary results on generalized weighted cylinders and almost cylindrical points. Section~\ref{sec:variationineq} proves the variational inequality for the \(\WW\)-functional near generalized cylinders. Section~\ref{sec:loja} establishes the Lojasiewicz inequality and derives the strong uniqueness theorem. Section~\ref{sec:singset} applies these results to the rectifiability of generalized cylindrical singular sets.

\vspace{1cm}

\textbf{Acknowledgments}: Hanbing Fang would like to thank his advisor, Prof. Xiuxiong Chen, for his encouragement and support. Hanbing Fang is supported by the Simons Foundation. Yu Li is supported by National Key Research and Development Program of China 2025YFA1018200, NSFC-12522105, YSBR-001 and research funds from University of Science and Technology of China and Chinese Academy of Sciences.

\section{Preliminaries}\label{sec:pre}
\subsection*{Weighted manifolds}
An $n$-dimensional \textbf{weighted Riemannian manifold} $(M^n,g,f)$ is a complete Riemannian manifold $(M^n,g)$ coupled with a smooth function $f:M\to \R$. For a weighted Riemannian manifold $(M^n,g,f)$, we define
\begin{align*}
	\mathbf{\VV}(g,f):=(4\pi)^{-\frac{n}{2}}\int_M e^{-f}\,\d V_g.
\end{align*}
Here, $\d V_g$ denotes the volume form of $(M,g)$. For simplicity, we set $\d V_f:=(4\pi)^{-\frac{n}{2}}e^{-f}\,\d V_g$. We say that $(M^n,g,f)$ is \textbf{normalized} if it satisfies the following normalization:
\begin{align}\label{equ:nor}
	\mathbf{\VV}(g,f)=1.
\end{align}

\begin{exmp}\label{exa:cylinder}
	Let $(N^{m},g_N)$ be an $m$-dimensional Einstein manifold with $\Ric(g_N)=g_N/2$. For any $n\geq \max\{m,3\}$, we define the \textbf{generalized weighted cylinders} as
	\begin{align*}
		\lc \mathcal{C}_N^{n-m}\rc_{-1}:=(\bar M,\bar g,\bar f)=\lc \R^{n-m}\times N^m, g_E\times g_N, \frac{|\vec x|^2}{4}+\frac{m}{2}+\Theta_{N,m}\rc,
	\end{align*}  
	where $g_E$ is the Euclidean metric on $\R^{n-m}$ and the vector $\vec x=(x_1,\ldots,x_{n-m})$ denotes the standard coordinate function on $\R^{n-m}$. The constant $\Theta_{N,m}$ is chosen to ensure that $\mathbf{\VV}(\bar g,\bar f)=1$. From a direct calculation, we can obtain that
	\begin{align*}
	\Theta_{N,m}=-\log\lc\frac{(4\pi)^{\frac{m}{2}}}{\vol(N)}\rc-\frac{m}{2}.
	\end{align*} 
\end{exmp}

\begin{defn}
	For a weighted Riemannian manifold $(M^n,g,f)$, we define:
	\begin{align*}
		\mathbf{\Phi}(g,f)&:=\frac{g}{2}-\Ric(g)-\na^2 f,\\
		\mathbf{\WW}(g,f)&:=\int_M2\Delta f-|\na f|^2+\scal+f-n\,\d V_f,\\
		\mathbf{\mu}(g,f)&:=2\Delta f-|\na f|^2+\scal+f-n.
	\end{align*}
\end{defn}

It is clear from the definition that $\mathbf{\Phi}(g,f)=0$ implies that $(M,g,f)$ is given by a Ricci shrinker. On the other hand, $\mathbf{\WW}(g,f)$ is Perelman's functional $\mathbf{\WW}(g,f,1)$ (see \cite{perelman2002entropy}). For the generalized weighted cylinder $\lc \CC_N^{n-m}\rc_{-1}=(\bar M,\bar g,\bar f)$ as in Example \ref{exa:cylinder}, we have
$$\mathbf{\Phi}(\bar g,\bar f)=0,\quad \mathrm{and} \quad \mathbf{\WW}(\bar g,\bar f)=\Theta_{N,m}.$$
In other words, $\Th_{N,m}$ denotes the entropy of $(\bar M,\bar g,\bar f)$ as a Ricci shrinker. For more information regarding the entropy of a Ricci shrinker, we refer readers to \cite[Section 5]{li2020heat}.

On a given weighted Riemannian manifold $( M^n, g, f)$, we can define the weighted Sobolev spaces, the weighted divergence operator, norms of tensors, weighted Laplacian and the stability operator $\LL$ as \cite[Definition 2.3]{fang2025strong}.

\subsection*{Preliminary results on generalized weighted cylinders}\label{sec:pregencyl}
In this subsection, we review preliminary results on generalized weighted cylinders $\lc \CC_N^{n-m}\rc_{-1}$ from \cite{li2023rigidity}. To begin with, let $(N^m,g_N)$ be an $m$-dimensional Einstein manifold with $\Ric(g_N)=g_N/2$. We remark that some statements in this subsection also hold for general compact Ricci shrinkers, but for simplicity, we only state them for $(N,g_N)$. 

For any Riemannian metric $g$ on $N$, Perelman's entropy $\mu(g,1)$ is defined as
\begin{align*}
\mu(g,1):=\inf_{\rho}\lc (4\pi)^{-\frac{n}{2}}\int_N \big(|\na \rho|^2+\scal+\rho-n\big) e^{-\rho}\,\d V_g\rc,
\end{align*}
where the infimum is taken for all smooth function $\rho$ satisfying $\int_N e^{-\rho}\,\d V_g=(4\pi)^{\frac{n}{2}}$. It was proved in \cite[Lemma 2.2]{sun2015kahler} that there exists a small neighborhood $\mathcal{U}$ of $g_N$ in $C^{2,\alpha}(S^2(N))$ such that for any $g\in \mathcal{U}$, the minimizer $\rho_g$ of $\mu(g,1)$ is unique and depends analytically on $g$.

\begin{defn}
	For any $g\in \mathcal{U}$, we define
	\begin{align*}
		\Phi(g)=\frac{g}{2}-\Ric(g)-\na^2_g \,\rho_g.
	\end{align*}
	In particular, $\Phi(g)=0$ if and only if $(N,g,\rho_g)$ is a Ricci shrinker with normalization $\scal(g)+|\na_g \rho_g|^2-\rho_g=\mu (g,1)$.
\end{defn}

By \cite[Lemma 2.6]{li2023rigidity}, on $(N,g_N)$, $\Phi'$ can be explicitly calculated by
\begin{align}\label{equ:1stPhi}
	\Phi'(h)=\frac{1}{2}\LL(h)+\Div^*\,\Div h+\na^2(\Tr_{g_N}(h)/2-f_1),
\end{align}
where $h\in C^{2,\alpha}(S^2(N))$ and $f_1$ is determined by
\begin{align*}
	(2\Delta+1)(\Tr_{g_N}(h)/2-f_1)=\Div\, \Div h.
\end{align*}
In particular, if $h\in \ker\Div$, then $f_1=\Tr_{g_N}(h)/2$ 
 and we have $\Phi'(h)=\frac{1}{2}\LL h$. Next, we recall the following definition.

\begin{defn}\label{def:ISD}
	The infinitesimal solitonic deformation space with respect to $(N,g_N)$ is defined as
	\begin{align*}
		\isd:=\left\{h\in C^\infty(S^2(N))|h\in \ker (\Phi')\cap \ker(\Div)\right\}.
	\end{align*}
\end{defn}

By \eqref{equ:1stPhi} and Definition \ref{def:ISD}, we immediately conclude that $\isd$ is a finite-dimensional linear space. Moreover, by \cite[Lemma 2.7]{li2023rigidity}, we have
\begin{align*}
	\ker(\Phi')=\isd\oplus \mathrm{Im}(\Div^*).
\end{align*}

\begin{defn}\label{def:h_k}
	Given $h_1\in\isd$, we say the sequence of symmetric $2$-tensors $\{h_i\}$ is induced by $h_1$ if $h_i\in (\ker\Phi')^{\perp}$ and for any $l\geq 2$,
	\begin{align*}
		\Phi'(h_l)+\pi_{(\ker\Phi')^{\perp}}(M_l(h_1,\cdots,h_{l-1}))=0,
	\end{align*}
	where $\pi_{(\ker\Phi')^{\perp}}$ is the projection onto $(\ker\Phi')^{\perp}$ in $L^2$.
\end{defn}

Here, the expressions $M_l(h_1,\cdots,h_{l-1})$ are given in \cite[Definition 2.13]{li2023rigidity}. For instance, we have
$$M_2(h_1)=\Phi^{(2)}(h_1,h_1)\quad \mathrm{and}\quad M_3(h_1,h_2)=3\Phi^{(2)}(h_1,h_2)+\Phi^{(3)}(h_1,h_1,h_1).$$

It follows from \cite[Lemma 2.7]{li2023rigidity} that there exists a unique sequence of $\{h_i\}$ induced by $h_1$. 
\begin{defn}\label{def:obs}
	For any $k\geq 2$, we say $(N,g_N)$ has an obstruction of order $k$ if there exists $\ep>0$ such that for any $h_1\in\isd$ with $\rVert h_1 \rVert_{L^2}<\ep$, 
	\begin{align*}
		\left\rVert\pi_{\isd}\bigg( \sum_{l=2}^k\frac{1}{l!}M_l(h_1,\cdots,h_{l-1})\bigg)\right\rVert_{L^2}\geq\ep \rVert h_1\rVert_{L^2}^k,
	\end{align*}
	where $\{h_i\}$ is the sequence induced by $h_1$. We say $(N,g_N)$ has an obstruction of order $1$ if $\isd=0$.
\end{defn}
For instance, if $(N,g_N)$ has an obstruction of order $3$, then for all $h_1\in\isd$ with $\rVert h_1\rVert_{L^2}$ small enough, 
\begin{align*}
	\left\rVert\pi_{\isd}\bigg( \frac{1}{6}\Phi^{(3)}(h_1,h_1,h_1)+\frac{1}{2}\Phi^{(2)}(h_1,h_2)+\frac{1}{2}\Phi^{(2)}(h_1,h_1)\bigg)\right\rVert_{L^2}\geq\ep \rVert h_1\rVert_{L^2}^3.
\end{align*}

In the next definition, we consider the sequence $\{f_i\}$ induced by $h_1\in\isd$ as in \cite[Definition 2.29]{li2023rigidity}.
\begin{defn}\label{def:f_k}
	Given $h_1\in\isd$, we say that the sequence of functions $\{f_i\}$ is induced by $h_1$ if $f_1=\Tr_{g_N}(h_1)/2$ and for any $l\geq 2$,
	\begin{align*}
		(2\Delta +1)f_l+L_l(h_1,\cdots, h_l,f_1,\cdots,f_{l-1})=0,
	\end{align*}
	where $\{h_i\}$ is the sequence induced by $h_1$ and $L_l$ is obtained from differentiating \emph{\cite[Equation (2.51)]{li2023rigidity}}.
\end{defn}

From now on, we assume that $(N^m, g_N)$ has obstruction of order \(3\) and satisfies the spectral condition
\[
    -\frac l2 \notin \operatorname{spec}\left(\LL_N|_{\mathrm{TT}}\right),
    \qquad \forall l\in \mathbb N^+.
\]
We consider $(\bar M,\bar g,\bar f)$ as in Example \ref{exa:cylinder}. By \cite[Proposition 4.2]{li2023rigidity}, we know that
\begin{align}\label{equ:ISDbarM}
	\isd=\isd_N\oplus \KK_0,
\end{align}
where $\isd,\isd_N$ are the spaces of infinitesimal solitonic deformations with respect to $(\bar M,\bar g,\bar f)$ and $(N,g_N)$; and $\KK_0$ is the linear space of quadratic Hermite polynomials on $\R^{n-m}$ (see \cite[Definition 2.4]{fang2025strong}). Regarding the definition of $\isd$ space in noncompact Ricci shrinker, we refer to \cite[Definition 3.18]{li2023rigidity}.

For any $\{g,f\}$ near $\{\bar g,\bar f\}$, we set $\{h,\chi\}=\{g-\bar g,f-\bar f\}$ and we assume that $h$ and $\chi$ are compactly supported. From \eqref{equ:ISDbarM}, we can write $\pi_{\isd}(h)=h_1+ug_N$ where $u\in\KK_0$ and $h_1\in \isd_N$. Let $\{h_i\}$ and $\{f_i\}$ be the sequences induced by $h_1$. Then we set $\zeta=h-h_1-h_2/2$ and $q=\chi-f_1-f_2/2$. Moreover, we define 
\begin{align*}
	\zeta':=\zeta-ug_N \quad \mathrm{and}\quad q':=q-\frac{m}{2}u.
\end{align*}

For simplicity, we set $\alpha:=\rVert u\rVert_{L^2},\,\beta:=\rVert h_1\rVert_{L^2}$ and $\BB(h,\chi)=(\BB_1(h,\chi),\cdots,\BB_{n-m}(h,\chi))$, where
\begin{align*}
		\BB_i(h,\chi):=\int_{\bar M}\left\la\partial_{x_i},\bar \na \lc\frac{1}{2}\Tr_{\bar g}(h)-\chi\rc\right\ra \,\d V_{\bar f}=\frac{1}{2}\int_{\bar M} x_i\lc\frac{1}{2}\Tr_{\bar g}(h)-\chi\rc\,\d V_{\bar f}.
\end{align*}

The next two rigidity inequalities from \cite[Proposition 4.9, Theorem 4.10]{li2023rigidity} will be of crucial importance: 
\begin{prop}\label{pro:rigidityineq1}
	There exists a constant $C=C(n,m,k,N)$ such that the following holds\textup{:}
	\begin{align*}
		\rVert\zeta'\rVert_{W^{k,2}}+\rVert\na q'\rVert_{W^{k-1,2}}\leq C\lc \rVert\Phi(g,f)\rVert_{W^{k-2,2}}+\rVert \Div_{\bar f}h\rVert_{W^{k-1,2}}+|\BB(h,\chi)|+\alpha^2+\alpha\beta+\beta^3\rc.
	\end{align*}
\end{prop}

\begin{prop}\label{pro:rigidityineq2}
	There exist constants $\delta,C,C_\ep$ depending on $n,m,N$, such that if $\rVert h\rVert_{C^2}+\rVert\chi\rVert_{C^2}\leq\delta$, then for any $\ep>0$, 
	\begin{align*}
		\alpha^2+\beta^3\leq& C\lc\rVert(1+|\vec{x}|^2)\Phi(g,f)\rVert_{L^1}+\beta(\rVert\Phi(g,f)\rVert_{L^2}+\rVert\Div_{\bar f}h\rVert_{W^{1,2}}+|\BB(h,\chi)|)\rc\nonumber\\
		&+C_\ep\lc \rVert\Phi(g,f)\rVert_{L^2}^{2-\ep}+\rVert\Div_{\bar f}h\rVert_{W^{1,2}}^{2-\ep}+|\BB(h,\chi)|^{2-\ep}\rc.
	\end{align*}
\end{prop}

\subsection*{Cylindrical and almost cylindrical points}\label{subsec:cylsing}
For basic conventions and results for closed Ricci flow, we refer to \cite[Section 2]{fang2025RFlimit}; for the construction and basic properties of Ricci flow limit spaces, we refer to \cite{fang2025RFlimit}. In this subsection, we focus on cylindrical flows. We consider the standard Ricci flow solution on the generalized cylinder as before:
\begin{align*}\index{$\mathcal C^k$}
\mathcal C^k_N:=(\bar M,(\bar g(t))_{t<0},(\bar f(t))_{t<0})=\left(\R^{k}\times N^{n-k}, g_E \times |t| g_{N}, \frac{|\vec{x}|^2}{4|t|}+\frac{n-k}{2}+\Theta_{N,n-k} \right).
\end{align*}
We denote by $d_{\mathcal C}^*$\index{$d_{\mathcal C}^*$} the spacetime distance on $\mathcal C^k_N$ as in \cite[Definition 3.5]{fang2025RFlimit}, with respect to a spacetime distance constant $\ep_0$ depending on $n$ and $Y$. Here, we implicitly assume $\Theta_{N,n-k} \ge -Y$.

Then, we set the completion of $\mathcal C^k_N$ under $d_{\mathcal C}^*$ by $\bar{\mathcal C}_N^k$\index{$\bar{\mathcal C}^k$}. It is straightforward to verify that the metric completion adds only the singular set $\R^k \times \{0\}$, which is the spine (see \cite[Definition D.4]{fang2025RFlimit}) of $\bar{\mathcal C}_N^k$. We then define the base point $p^*$ as the limit of $(\bar p, t)$ as $t \nearrow 0$ with respect to $d_{\mathcal C}^*$, where $\bar p \in \bar M$ is a minimum point of $\bar f(-1)$. It is clear that $p^*$ is independent of the choice of $\bar p$. Moreover, for any $t<0$,
	\begin{align*}
\nu_{p^*;t}=(4\pi |t|)^{-\frac n 2} e^{-\bar f(t)} \,\mathrm{d}V_{\bar g(t)}.
	\end{align*}

In general, let $\XX=\{M^n, (g(t))_{t \in [-T,0)}\}$ be a closed Ricci flow with entropy bounded below by $-Y$, where $0$ is the first singular time. Suppose $(Z, d_Z, \t)$ is the completion of $\XX$. Fix a point $z \in Z_0$. Then the potential function $f_z$ is smooth for $t<0$. Moreover, we write
	\begin{align*}
\nu_{z;t}=(4\pi |t|)^{-\frac n 2} e^{-f_z} \,\mathrm{d}V_{g(t)}.
	\end{align*}
Then we can define the modified Ricci flow as in \eqref{equ:MRFintroI}.

Next, we consider a general noncollapsed Ricci flow limit space $(Z, d_Z, \t)$ over $\III$, obtained as the limit of a sequence in $\mathcal M(n, Y, T)$ (see \cite[Section 3]{fang2025RFlimit}). Then, we have a definition similar to \cite[Definition 2.22]{fang2025strong}.

\begin{defn}
A point $z \in Z_{\III^-}$ is called an \textbf{$N$-cylindrical point} if one tangent flow at $z$ (see \emph{\cite[Definition 7.8]{fang2025RFlimit}}) is isometric to $\bar{\mathcal C}_N^k$.
\end{defn}

Note that by \cite[Remark 2.24]{fang2025strong}, the tangent flow at an $N$-cylindrical point is unique under our assumption on $N$. Next, we introduce the following definition.

\begin{defn}\label{def:almost0}
Let $(Z, d_Z, \t)$ be a noncollapsed Ricci flow limit space arising as the pointed Gromov--Hausdorff limit of a sequence in $\MM(n, Y, T)$. A point $z \in Z_{\III^-}$ is called \textbf{$(k,\ep,r)$-$N$-cylindrical}\index{$(k,\ep,r)$-cylindrical} if $\t(z)-\ep^{-1} r^2 \in \III^-$ and
	  \begin{align*}
(Z, r^{-1} d_Z, z, r^{-2}(\t-\t(z))) \quad \text{is $\ep$-close to} \quad (\bar{\mathcal C}^k_N ,d^*_{\mathcal C}, p^*,\t) \quad \text{over} \quad [-\ep^{-1}, \ep^{-1}].
  \end{align*} 

Let $\tilde \phi$ be an $\ep$-map as defined in \emph{\cite[Definition 5.37]{fang2025RFlimit}}, which is from $B^*(p^*,\ep^{-1}) \cap \left(\bar{\mathcal C}^k_N\right)_{[-\ep^{-1}, \ep^{-1}]}$ to $Z_{[\t(z)-\ep^{-1} r^2, \t(z)+\ep^{-1} r^2]}$, where $B^*(p^*,\ep^{-1})$ is the metric ball in $\bar{\mathcal C}_N^k$ with respect to $d_{\mathcal C}^*$. Then, we define\emph{:}
\begin{align*}
\LL_{z,r}:=\tilde \phi \lc B^*(p^*,\ep^{-1}) \cap \left(\bar{\mathcal C}^k_N\right)_0 \rc,\index{$\LL_{z,r}$}
\end{align*}
	and say that $z$ is \textbf{$(k,\ep,r)$-$N$-cylindrical with respect to $\LL_{z,r}$}.
\end{defn}

\section{Variations of the \texorpdfstring{$\WW$}{WW}-functional near generalized weighted cylinders}\label{sec:variationineq}

In this section, we derive the variational estimates for $\WW$. Let $(N^{m},g_N)$ be a closed Einstein manifold with $\Ric(g_N)=g_N/2$, which has obstruction of order \(3\) and satisfies the spectral condition
\[
    -\frac l2 \notin \operatorname{spec}\left(\LL_N|_{\mathrm{TT}}\right),
    \qquad \forall l\in \mathbb N^+.
\]
We will consider the generalized weighted cylinder $\lc \CC^{n-m}_N\rc_{-1}=(\bar M,\bar g,\bar f)$ as defined in Example \ref{exa:cylinder}. We follow the notations as in Section \ref{sec:pre}. The constant $C$ appearing in this section depends on $n,m, N$ and may vary from line to line. We denote by $C(A,B,\ldots)$ the constants depending on $A,B$, etc.

For any $\{g,f\}$ near $\{\bar g,\bar f\}$, set $\{h,\chi\}=\{g-\bar g, f-\bar f\}$ and we assume that $h$ and $\chi$ are compactly supported. Assume 
\begin{align*}
	\delta:=\rVert h\rVert_{C^2}+\rVert \chi\rVert_{C^2}\ll 1.
\end{align*}
As in Section \ref{sec:pregencyl}, we set $\pi_{\isd}(h)=h_1+ug_N$ and $\{h_i\},\,\{f_i\}$ to be two sequences of $2$-tensors and functions induced by $h_1$ (see Definitions \ref{def:h_k} and \ref{def:f_k}).  Write
\begin{align}\label{equ:defuh_1}
h=\zeta'+ug_N+h_1+\frac{1}{2} h_2,\quad \chi=q'+\frac{m}{2}u+f_1+\frac{1}{2} f_2,	
\end{align}
where $ug_N\in \KK_0 g_N$ is a quadratic Hermite polynomial on $\R^{n-m}$.  Note that for $u$, the following estimate holds (see \cite[Lemma 5.17]{colding2021singularities}):
\begin{lem}
	There exists a constant $C=C(n)$ such that the following holds for $u\in \KK_0$ on $\R^{n-m}$\textup{:}
	\begin{equation*}
		|u|+|\vec{x}||\nabla u|+(1+|\vec{x}|^2)| \na^2u|\leq C (1+|\vec{x}|^2) \rVert u\rVert_{L^2}.
	\end{equation*}
\end{lem}

We also need the following estimate from \cite[Lemma 3.6]{li2023rigidity}:
\begin{lem}\label{lem:coninequ}
	For any $k\geq 0$, there exists a constant $C_k=C_k(n)$ such that for any tensor $T\in W^{k,2}(T^{r,s}\bar M)$,
	\begin{align*}
		\int_{\bar M}\bar f^k|T|^2\,\d V_{\bar f}\leq C_k\rVert T\rVert_{W^{k,2}}^2.
	\end{align*}
\end{lem}

As in \cite[Section 4]{li2023rigidity}, we set 
\begin{align*}
	g(s,t)=\bar g+s\zeta +th_1+\frac{t^2}{2}h_2,\quad f(s,t)=\bar f+sq+tf_1+\frac{t^2}{2}f_2,
\end{align*}
where $\zeta=\zeta'+ug_N,\, q=q'+\frac{m}{2}u$ and
\begin{align*}
	g^0(s,t)=\bar g+sug_N +th_1+\frac{t^2}{2}h_2,\quad f^0(s,t)=\bar f+\frac{m}{2}su+tf_1+\frac{t^2}{2}f_2.
\end{align*}
Then $g=g(1,1), f=f(1,1)$. For simplicity, we set
\begin{align*}
	\begin{dcases}
		& \d V_{s,t}:=\d V_{f(s,t)},\\
		&V(s,t)=\int_{\bar M} 1 \,\d V_{s,t},\\
		&Q(s,t)=\frac{1}{2} \Tr_{g(s,t)} \zeta-q,\\
		&P(s,t):=\frac{1}{2}\Tr_{g(s,t)}(h_1+th_2)-(f_1+tf_2),\\
		&w(s,t):= 2\Delta_{g(s,t)} f(s,t)-|\nabla_{g(s,t)} f(s,t)|^2+\scal (g(s,t))+f(s,t)-n,\\
		&W(s,t)=\mathbf{\WW}(g(s,t),f(s))= \int_{\bar M}  w(s,t) \,\d V_{s,t},\\
		&\Phi(s,t)=\mathbf{\Phi}(g(s,t), f(s,t))=\frac{g(s,t)}{2}-\Ric(g(s,t))-\na^2_{g(s,t)} f(s,t),\\
		&\Phi^0(s,t)=\mathbf{\Phi}(g^0(s,t),f^0(s,t))=\frac{g^0(s,t)}{2}-\Ric(g^0(s,t))-\na^2_{g^0(s,t)} f^0(s,t).
	\end{dcases}
\end{align*}
For the rest of this section, all geometric quantities such as $\Ric$, $\na$, $\na^2$, and $\Div_f$ are defined with respect to the pair $(g(s,t), f(s,t))$, though we omit the subscript $g(s,t)$ and the parameter $s$ and $t$ for simplicity. The norms $\|\cdot\|_{L^1}$, $\|\cdot\|_{L^2}$, $|\cdot|$ and $[\cdot]_l$ are taken with respect to $\d V_{s,t}$ and $g(s,t)$, respectively. Notice that since $\|h\|_{C^2}+\|\chi\|_{C^2}$ is sufficiently small, the volume form $\d V_{s,t}$ is uniformly comparable to $\d V_{0,0}=\d V_{\bar f}$, and the norms $[\cdot]_l$ are uniformly comparable to the corresponding norms defined with respect to $\bar g$ for $l \le 2$.

As in Section \ref{sec:pregencyl}, we define $\alpha:=\rVert u\rVert_{L^2}$ and $\beta:=\rVert h_1\rVert_{L^2}$. Then by \cite[Lemma 4.5]{li2023rigidity}, we have
\begin{align}\label{equ:basicesti1}
	\alpha+\beta\leq C\delta, \quad \rVert h_1\rVert_{C^5}+\rVert f_1\rVert_{C^5}\leq C\beta, \quad \rVert h_2\rVert_{C^5}+\rVert f_2\rVert_{C^5}\leq C\beta^2.
\end{align}

For simplicity, we define
\begin{align}\label{equ:defmathfrakX}
\mathfrak{X}:=\|\Phi(g,f)\|_{L^2}+\|\Div_{\bar f} h\|_{W^{1,2}}+|\mathcal B(h,\chi)|	.
\end{align}
Then we can easily get 
$$\mathfrak{X}\leq C\delta,\quad [\zeta]_2+[q]_2\leq C(\delta+\beta)\leq C\delta,\quad [\zeta']_2+[q']_2\leq C\delta(1+|\vec x|^2).$$
And it follows from Propositions \ref{pro:rigidityineq1}, \ref{pro:rigidityineq2} that
\begin{cor}\label{cor:rigidineq}
	We have
	\begin{align}\label{basicesti2}
		\alpha^2+\beta^3\leq C\mathfrak{X},\quad \rVert\zeta'\rVert_{W^{2,2}}+\rVert\na q'\rVert_{W^{1,2}}\leq C\mathfrak{X}^{\frac{5}{6}}.
	\end{align}
\end{cor}

The next proposition follows immediately from Definitions \ref{def:ISD}, \ref{def:h_k} and \ref{def:f_k}: 
\begin{prop}
	$h_1,h_2,f_1,f_2$ satisfy that at $s=t=0$,
	\begin{align*}
		\mathbf{\Phi}^{(1)}((h_1,f_1))=0,\quad \Div_{\bar f}h_1=0,\quad \mathbf{\Phi}^{(1)}((h_2,f_2))+\pi_{(\ker\Phi')^{\perp}}\mathbf{\Phi}^{(2)}((h_1,f_1),(h_1,f_1))=0,
	\end{align*}
	and
	\begin{align*}
		f_1=\frac{1}{2}\Tr_{\bar g}h_1(\Leftrightarrow P(0,0)=0),\quad (2\Delta_{\bar f}+1)f_2+L_2(h_1,h_2,f_1)=0.
	\end{align*}
\end{prop}

In the remaining part of this section, we will calculate and estimate derivatives of $W$, following similar ideas as in \cite[Section 4]{fang2025strong}.
Denote for simplicity, 
\begin{align*}
	\bar W:=W(0,0)(=w(0,0)),\quad W:=W(1,1),\quad \bar Q:=\int_{\bar M} Q(0,0)\,\d V_{0,0}.
\end{align*}

Note first that at $s=t=0$, we have
\begin{align}\label{phideri1}
	\Phi_s^0(0,0)=\Phi_t^0(0,0)=\Phi_t(0,0)=0,\quad \Phi_s(0,0)=\frac{1}{2}\LL\zeta+\Div_{\bar f}^*\, \Div_{\bar f}\zeta+\na^2Q(0,0).
\end{align}
And for any $s,t\in [0,1]$, we have $\Phi_s=\mathbf{\Phi}^{(1)}\big((\zeta,q)\big), \Phi_t=\mathbf{\Phi}^{(1)}\big((h_1,f_1)\big)$ and
\begin{align*}
	 \Phi_{ss}=\mathbf{\Phi}^{(2)}\left((\zeta,q),(\zeta,q)\right),\quad \Phi_{st}=\mathbf{\Phi}^{(2)}\big((\zeta,q),(h_1+th_2,f_1+tf_2)\big),\\ \Phi_{tt}=\mathbf{\Phi}^{(2)}\big((h_1+th_2,f_1+tf_2),(h_1+th_2,f_1+tf_2)\big)+\mathbf{\Phi}^{(1)}\big((h_2,f_2)\big).
\end{align*}
Here we use lower index to indicate derivatives of $\Phi(s,t)$ with respect to $s,\,t$ and we use upper index to indicate derivatives of $\Phi(g,f)$ with respect to $(g,f)$, e.g., 
\begin{align*}
	\mathbf{\Phi}^{(2)}\left((\zeta,q),(\zeta,q)\right)|_{(g,f)}:=\frac{\d ^2}{\d s^2}\bigg|_{s=0} \mathbf{\Phi}(g+s\zeta, f+s q).
\end{align*}

 The next lemma calculates the first derivatives.
 \begin{lem}\label{lem:1stderi}
 	For any $s,t\in [0,1]$, we can calculate the first derivatives as\textup{:}
 	\begin{enumerate}[label=\textnormal{(\roman{*})}]
 		\item 
 		$w_s=-2(\Delta_{ f}+\frac{1}{2})Q+\la\Phi,\zeta\ra+\Div_f\,\Div_f\zeta$.
 		\item $w_t=-2(\Delta_f+\frac{1}{2})P+\la \Phi, h_1+th_2\ra +\Div_f\,\Div_f (h_1+th_2)$.
 		\item $(\d V_{s,t})_s=Q\,\d V_{s,t},\, (\d V_{s,t})_t=P\,\d V_{s,t}$.
 		\item $	W_s=\int_{\bar M}\la \Phi,\zeta\ra \,\d V_{s,t}+\int_{\bar M}Q(w-1)\,\d V_{s,t}$.
 		\item $W_t=\int_{\bar M}\la \Phi,h_1+th_2\ra \,\d V_{s,t}+\int_{\bar M}P(w-1)\,\d V_{s,t}$.
 		\item $P_s=Q_t=-\frac{1}{2}\la \zeta,h_1+th_2\ra.$
 		\item $Q_s=-\frac{1}{2}|\zeta|^2,\, P_t=\frac{1}{2}\Tr_{g(s,t)} h_2-\frac{1}{2}|h_1+th_2|^2-f_2$.
 	\end{enumerate}
 \end{lem}
 \begin{proof}
 	Item (i) and (ii) follow from \cite[Lemma 4.9]{fang2025strong} and item (iii), (vi) and (vii) follow from direct calculations. For item (iv) and (v), notice that
 $$W_s=\int_{\bar M}w_s\,\d V_{s,t}+\int_{\bar M}w \,(\d V_{s,t})_s,\quad W_t=\int_{\bar M}w_t\,\d V_{s,t}+\int_{\bar M}w \,(\d V_{s,t})_t. $$
Thus by integration by parts using item (i) and (iii), we have
\begin{align*}
	W_s&=\int_{\bar M}-2(\Delta_{ f}+\frac{1}{2})Q+\la\Phi,\zeta\ra+\Div_f\Div_f\zeta\, \d V_{s,t}+\int_{\bar M} wQ\,\d V_{s,t}\\
	&=\int_{\bar M}\la \Phi,\zeta\ra \,\d V_{s,t}+\int_{\bar M}Q(w-1)\,\d V_{s,t}.
\end{align*}
A similar formula also holds for $W_t$. This completes the proof.
 \end{proof}

When evaluating at $(s,t)=(0,0)$, we can obtain:
\begin{cor}\label{cor:1stderiat0}
	The following holds at $(s,t)=(0,0)$\textup{:}
	\begin{enumerate}[label=\textnormal{(\roman{*})}]
		\item $w_s=-2(\Delta_{\bar f}+\frac{1}{2})Q(0,0)+\Div_{\bar f}\, \Div_{\bar f} \zeta,\, w_t= \Div_{\bar f}\, \Div_{\bar f} h_1=0$.
		
		\item $(\d V_{s,t})_s(0,0)=Q(0,0)\,\d V_{0,0},\, (\d V_{s,t})_t(0,0)=0$.
		
		\item $W_s(0,0)=(\bar W-1)\bar Q,\, W_t(0,0)=0$.
		
		\item $P_s=Q_t=-\frac{1}{2}\la\zeta,h_1\ra$.
		
		\item  $Q_s=-\frac{1}{2}|\zeta|^2,\, P_t=\frac{1}{2}\Tr_{\bar g}h_2-\frac{1}{2}|h_1|^2-f_2$.
	\end{enumerate}
	
\end{cor}

Then we consider second derivatives of $P$ and $Q$. In the remaining of this section, we always do the calculations under normal coordinates of $g(s,t)$.
\begin{lem}\label{lem:2ndderi1}
	For any $s,t\in [0,1]$, we have
	\begin{enumerate}[label=\textnormal{(\roman{*})}]
		\item $P_{ss}=Q_{st}=\zeta_{ij}(h_1+th_2)_{jl}\zeta_{il},\, P_{st}=Q_{tt}=-\frac{1}{2}\la \zeta,h_2\ra+\zeta_{ij}(h_1+th_2)_{jl}(h_1+th_2)_{il}$.
		
		\item $P_{tt}=-\frac{3}{2}(h_1+th_2)_{ij}(h_2)_{ij}+(h_1+th_2)_{ij}(h_1+th_2)_{jl}(h_1+th_2)_{il}$.
		
		\item $Q_{ss}=\zeta_{ij}\zeta_{jl}\zeta_{il}$.
	\end{enumerate}
\end{lem}

As a direct corollary of Lemmas \ref{lem:1stderi} and \ref{lem:2ndderi1}, we have:
\begin{cor}\label{estihighderiPQ}
	For any $s,t\in [0,1]$, we have\textup{:}
	\begin{enumerate}[label=\textnormal{(\roman{*})}]
		\item $[P_s]_0\leq C\beta\big([\zeta']_0+(1+|\vec x|^2)\alpha\big),\,[P_t]_0\leq C\beta^2$.
		\item $[Q_s]_0\leq C\big([\zeta']_0^2+(1+|\vec x|^4)\alpha^2\big),\,[Q_t]_0\leq C\beta\big([\zeta']_0+(1+|\vec x|^2)\alpha\big)$.
		\item $[P_{ss}]_0+[Q_{st}]_0\leq C\beta\big([\zeta']_0^2+(1+|\vec x|^4)\alpha^2\big)$.
		\item $[P_{st}]_0+[Q_{tt}]_0\leq C\beta^2\big([\zeta']_0+(1+|\vec x|^2)\alpha\big)$.
		\item $[P_{tt}]_0\leq C\beta^3,\, [Q_{ss}]_0\leq C\big([\zeta']_0^3+(1+|\vec x|^6)\alpha^3\big)$.
	\end{enumerate}
\end{cor}

\begin{lem}\label{lem:Qesti1}
	For any $s,t\in [0,1]$, we have
	\begin{equation*}
		\int_{\bar M} Q^2(s,t)\,\d V_{s,t}\leq C(\alpha^4+\beta^4+\alpha^2\mathfrak{X}^{\frac{5}{3}}+\beta^2 \mathfrak{X}^{\frac{5}{3}}+\delta^2 \mathfrak{X}^{\frac{5}{3}})+4\int_{\bar M} Q^2(0,0)\,\d V_{0,0}.
	\end{equation*}
\end{lem}
\begin{proof}
	Let $\eta(s,t):= \int_{\bar M} Q^2(s,t)\,\d V_{s,t}$. We can calculate
	\begin{align}\label{Qineq1}
		\eta_s=\int_{\bar M} 2QQ_s\,\d V_{s,t}+\int_{\bar M} Q^3\,\d V_{s,t}\leq -\int_{\bar M} |Q(s)||\zeta|^2\,\d V_{s,t}+ C\delta \eta.
	\end{align}
	Here, in the last inequality, we have used the fact $|Q|\leq C\delta$. Note that
	$$|\zeta|^4\leq C(u^4+|\zeta'|^4)\leq Cu^4+C|\zeta'|^2(u^2+|h|^2+\beta^2)\leq Cu^4+C\alpha^2(1+|\vec x|^2)|\zeta'|^2+C\delta^2 |\zeta'|^2+C\beta^2|\zeta'|^2,$$
	thus it follows from Lemma \ref{lem:coninequ} and Corollary \ref{cor:rigidineq} that 
	\begin{align}\label{equ:zeta4}
	\int_{\bar M}|\zeta|^4\,\d V_{s,t}\leq C(\alpha^4+\alpha^2\mathfrak{X}^{\frac{5}{3}}+\beta^2 \mathfrak{X}^{\frac{5}{3}}+\delta^2\mathfrak{X}^{\frac{5}{3}}).	
	\end{align}
So by Cauchy-Schwarz inequality, we have
	\begin{align}\label{Qineq2}
		\int_{\bar M} |Q(s)||\zeta|^2\,\d V_{s,t}\leq C(\alpha^2+\alpha\mathfrak{X}^{\frac{5}{6}}+\beta \mathfrak{X}^{\frac{5}{6}}+\delta \mathfrak{X}^{\frac{5}{6}})\eta^{1/2}(s).
	\end{align}
	Combining \eqref{Qineq1} with \eqref{Qineq2}, we have
	\begin{align*}
		\eta_s\leq C(\alpha^2+\alpha\mathfrak{X}^{\frac{5}{6}}+\beta \mathfrak{X}^{\frac{5}{6}}+\delta \mathfrak{X}^{\frac{5}{6}})\eta^{1/2}(s)+C\delta\eta.
	\end{align*}
	
	Now if we denote by $\psi(s,t)=(e^{-C\delta s}\eta)^{1/2}$, we get
	\begin{align*}
		\psi_s\leq \frac{C}{2}e^{-\frac{C \delta s}{2}}(\alpha^2+\alpha\mathfrak{X}^{\frac{5}{6}}+\beta \mathfrak{X}^{\frac{5}{6}}+\delta \mathfrak{X}^{\frac{5}{6}}).
	\end{align*}
	Thus by modifying the constant $C$, it follows that
	$$\psi(s,t)\leq C(\alpha^2+\alpha\mathfrak{X}^{\frac{5}{6}}+\beta \mathfrak{X}^{\frac{5}{6}}+\delta \mathfrak{X}^{\frac{5}{6}})+\psi(0,t)=C(\alpha^2+\alpha\mathfrak{X}^{\frac{5}{6}}+\beta \mathfrak{X}^{\frac{5}{6}}+\delta \mathfrak{X}^{\frac{5}{6}})+\eta^{1/2}(0,t).$$
	Using Young's inequality, we get
	\begin{align}\label{equ:Qesti2}
		\eta(s,t)\leq C(\alpha^4+\alpha^2\mathfrak{X}^{\frac{5}{3}}+\beta^2 \mathfrak{X}^{\frac{5}{3}}+\delta^2 \mathfrak{X}^{\frac{5}{3}})+2\eta(0,t).
	\end{align}
	
	On the other hand, by the same calculations as above and using Lemma \ref{lem:1stderi}, we have
	\begin{align*}
		\eta_t=-\int_{\bar M}Q\la h_1+th_2,\zeta\ra \,\d V_{s,t}+\int_{\bar M}Q^2P\,\d V_{s,t}\leq C\lc\int_{\bar M}|h_1+th_2|^2|\zeta|^2\,\d V_{s,t}\rc^{1/2}\eta^{1/2}+C\delta \eta.
	\end{align*}
	Using the fact that
	$$\int_{\bar M}|h_1+th_2|^2|\zeta|^2\,\d V_{s,t}\leq C\beta^2\int_{\bar M}u^2+|\zeta'|^2\,\d V_{s,t}\leq C\beta^2(\alpha^2+\mathfrak{X}^{\frac{5}{3}}),$$
	we obtain
	\begin{align*}
		\eta_t\leq C\beta(\alpha+\mathfrak{X}^{\frac{5}{6}})\eta^{1/2}+C\delta\eta.
	\end{align*}
	Integrating as in obtaining \eqref{equ:Qesti2}, it follows that
	\begin{align}\label{equ:Qesti1}
		\eta(s,t)\leq C\beta^2(\alpha^2+\mathfrak{X}^{\frac{5}{3}})+2\eta(s,0).
	\end{align}
	
	Combining \eqref{equ:Qesti2} with \eqref{equ:Qesti1}, we get
	\begin{align*}
		\eta(s,t)\leq C(\alpha^4+\beta^4+\alpha^2\mathfrak{X}^{\frac{5}{3}}+\beta^2 \mathfrak{X}^{\frac{5}{3}}+\delta^2 \mathfrak{X}^{\frac{5}{3}})+4\eta(0,0).
	\end{align*}
	This completes the proof.
\end{proof}

By the same argument, the following integral estimate for $P$ holds (note that $P(0,0)=0$):
\begin{prop}\label{pro:Pesti}
	For any $s,t\in [0,1]$, we have
	\begin{equation*}
		\int_{\bar M} P^2(s,t)\,\d V_{s,t}\leq C(\alpha^2\beta^2+\beta^2 \mathfrak{X}^{\frac{5}{3}}+\beta^4).
	\end{equation*}
\end{prop}

\begin{prop}\label{pro:Qesti2}
	For any $s,t\in [0,1]$, we have
	\begin{equation*}
		\int_{\bar M} Q^2(s,t)\,\d V_{s,t}\leq C(\alpha^4+\beta^4+\mathfrak{X}^{\frac{5}{3}}+|V(1,1)-V(0,0)|^2).
	\end{equation*}
\end{prop}
\begin{proof}
	For $V(s,t)=\int_{\bar M} \,\d V_{s,t}$, by direct calculations, we have
	$$V_s=\int_{\bar M} Q\,\d V_s, \quad V_t=\int_{\bar M}P\,\d V_{s,t},$$
	$$ V_{ss}=\int_{\bar M} Q_s+Q^2\,\d V_{s,t},\quad V_{st}=\int_{\bar M}Q_t+QP\,\d V_{s,t},\quad V_{tt}=\int_{\bar M}P_t+P^2\,\d V_{s,t}.$$
	
	By Corollary \ref{cor:1stderiat0}, it follows that
	\begin{align*}
		V_s(0,0)=\bar Q, \quad V_t(0,0)=0, \quad V_{ss}(0,0)=\int_{\bar M}-\frac{1}{2}|\zeta|^2+Q^2(0,0)\,\d V_{0,0},
	\end{align*}
	and 
	\begin{align}\label{deriVat02}
		V_{tt}(0,0)=\int_{\bar M}\frac{1}{2}\Tr_{\bar g}h_2-f_2-\frac{1}{2}|h_1|^2\,\d V_{0,0},\quad V_{st}(0,0)=\int_{\bar M}-\frac{1}{2}\la h_1,\zeta\ra\,\d V_{0,0}.
	\end{align}
	
	By Taylor expansion of $V$, we have
	\begin{equation}\label{taylorofV1}
		\left|V(1,1)-V(0,0)-V_s-V_t-\frac{1}{2}(V_{ss}+2V_{st}+V_{tt})\right|\leq \frac{4}{3}\sup_{s,t\in [0,1]}|V^{(3)}(s,t)|,
	\end{equation}
	where on the left hand side, all derivatives are evaluated at $(0,0)$ and $V^{(3)}$ denotes third derivatives of $V$. Note that from Corollary \ref{cor:rigidineq},
	$$\int_{\bar M}|\zeta|^2\,\d V_{0,0}\leq C(\alpha^2+\mathfrak{X}^{\frac{5}{3}}),$$
	so by \eqref{deriVat02}, we have
	\begin{align*}
		&|V_{st}(0,0)|\leq C\left|\int_{\bar M}\la h_1,\zeta\ra \,\d V_{0,0}\right| \leq C\beta (\alpha+\mathfrak{X}^{\frac{5}{6}}),\quad |V_{tt}(0,0)|\leq C\beta^2.
	\end{align*}
	Therefore, it follows from \eqref{taylorofV1} that
	\begin{equation}\label{taylorofV2}
		\left|\bar Q+\frac{1}{2}\int_{\bar M} Q^2(0,0)\,\d V_{0,0}\right|\leq C(\alpha^2+\beta^2+\mathfrak{X}^{\frac{5}{3}})+\frac{4}{3}\sup_{s,t\in [0,1]}|V^{(3)}(s,t)|+|V(1,1)-V(0,0)|.
	\end{equation}
	
	Now we estimate the third derivatives. Notice that for any $s,t\in [0,1]$, $|P|+|Q|\leq C\delta$, and 
	\begin{align}
		\label{equ:3rdQderi1}&V_{sss}=\int_{\bar M}Q_{ss}+3QQ_s+Q^3\,\d V_{s,t},\\ \label{equ:3rdQderi2}&V_{sst}=\int_{\bar M}Q_{st}+Q_sP+2Q_tQ+Q^2P\,\d V_{s,t},\\
		\label{equ:3rdQderi3} &V_{stt}=\int_{\bar M}Q_{tt}+2Q_tP+QP_t+QP^2\,\d V_{s,t},\\
		\label{equ:3rdQderi4} &V_{ttt}=\int_{\bar M}P_{tt}+3PP_t+P^3\,\d V_{s,t}.
	\end{align}
	By Proposition \ref{pro:Pesti}, $\rVert P\rVert_{L^2}\leq C(\alpha^2+\beta^2+\alpha\mathfrak{X}^{\frac{5}{6}}+\beta\mathfrak{X}^{\frac{5}{6}}+\delta\mathfrak{X}^{\frac{5}{6}})$. By Corollary \ref{estihighderiPQ} and Lemma \ref{lem:Qesti1}, we obtain
	\begin{align*}
		&\int_{\bar M}P^3+Q^3+|P|Q^2+|Q|P^2\, \d V_{s,t}\leq C\delta \int_{\bar M}Q^2+P^2\, \d V_{s,t}\\
		\leq& C\delta \left(\alpha^4+\beta^4+\alpha^2\mathfrak{X}^{\frac{5}{3}}+\beta^2 \mathfrak{X}^{\frac{5}{3}}+\delta^2 \mathfrak{X}^{\frac{5}{3}}+\int_{\bar M}Q^2(0,0)\,\d V_{0,0}\right), 
	\end{align*}
	and
	\begin{align*}
		\int_{\bar M}|Q_s Q|\, \d V_{s,t}&\leq C\int_{\bar M} \lc [\zeta']_0^2+(1+|\vec x|^4)\alpha^2\rc |Q|\, \d V_{s,t}\leq C\delta(\mathfrak{X}^{\frac{5}{3}}+\alpha^2),
		\end{align*}
			and similarly, 
	\begin{align*}
		 \int_{\bar M}|Q_tQ|+|P_tQ|+|Q_sP|+|Q_tP|+|P_tP|\,\d V_{s,t}&\leq C(\alpha^2+\beta^2+\alpha\mathfrak{X}^{\frac{5}{6}}+\beta\mathfrak{X}^{\frac{5}{6}}+\delta\mathfrak{X}^{\frac{5}{6}}),\\
		\int_{\bar M}|Q_{ss}|+|Q_{st}|+|Q_{tt}|+|P_{tt}|\,\d V_{s,t}&\leq C(\alpha^2+\beta^2+\mathfrak{X}^{\frac{5}{3}}).
	\end{align*}
	Thus 
	\begin{align*}
		\sup_{s,t\in [0,1]}|V^{(3)}(s,t)|\leq C(\alpha^2+\beta^2+\mathfrak{X}^{\frac{5}{3}})+C\delta\int_{\bar M}Q^2(0,0)\,\d V_{0,0},
	\end{align*}
	which, combining with \eqref{taylorofV2}, gives
	\begin{align}\label{taylorofV3}
		\left|\bar Q+\frac{1}{2}\int_{\bar M} Q^2(0,0)\,\d V_{0,0}\right|\leq C(\alpha^2+\beta^2+\mathfrak{X}^{\frac{5}{3}})+C\delta\int_{\bar M}Q^2(0,0)\,\d V_{0,0}+|V(1,1)-V(0,0)|.
	\end{align}
	By Poincar\'e's inequality on cylinder,\,\eqref{basicesti2} and the fact $|\na Q(0,0)|\leq C([\zeta']_1+[\na q']_0)$, we have
	\begin{align}\label{taylorofV4}
		0\leq\int_{\bar M} Q^2(0,0)\,\d V_{0,0}-\lc\int_{\bar M} Q(0,0)\,\d V_{0,0}\rc^2\leq C\int_{\bar M} |\nabla Q(0,0)|^2\,\d V_{0,0}\leq C\mathfrak{X}^{\frac{5}{3}}.
	\end{align}
	Combining \eqref{taylorofV3} and \eqref{taylorofV4}, we get
	$$|\bar Q|\leq C(\alpha^2+\beta^2+\mathfrak{X}^{\frac{5}{3}})+\left(\frac{1}{2}+C\delta\right)\bar Q^2+|V(1,1)-V(0,0)|.$$
	Since $|\bar Q | \leq C\delta$, we conclude that if $\delta\ll 1$, then
	\begin{equation*}
		|\bar Q|= \left|\int_{\bar M} Q(0,0)\,\d V_{0,0}\right| \leq C(\alpha^2+\beta^2+\mathfrak{X}^{\frac{5}{3}}+|V(1,1)-V(0,0)|).
	\end{equation*}
	By \eqref{taylorofV4} again, it follows that
	\begin{align}\label{taylorofV5}
		\int_{\bar M} Q^2(0,0)dV_{0,0} \leq C(\alpha^4+\beta^4+\mathfrak{X}^{\frac{5}{3}}+|V(1,1)-V(0,0)|^2).
	\end{align}
	Therefore Proposition \ref{pro:Qesti2} follows from Lemma \ref{lem:Qesti1} and \eqref{taylorofV5}. This completes the proof.
\end{proof}

\begin{lem}\label{lem:PQesti2}
	For any $s,t\in [0,1]$, we have
	\begin{align*}
		\int_{\bar M}(1+|\vec x|^2)(Q^2(s,t)+P^2(s,t))\,\d V_{s,t}\leq C(\delta \alpha^3+\beta^4+\mathfrak{X}^{\frac{5}{3}}+|V(1,1)-V(0,0)|^2).
	\end{align*}
\end{lem}
\begin{proof}
	The proof could follow from the same argument as in \cite[Lemma 4.13]{fang2025strong}. Here we provide a slightly different argument. First note that by Lemma \ref{lem:coninequ} and Proposition \ref{pro:Qesti2}, we have
	\begin{align}\label{equ:Qesti11}
		\int_{\bar M}(1+|\vec x|^2)Q^2(0,0)\,\d V_{0,0}&\leq C\int_{\bar M}|Q(0,0)|^2+|\na Q(0,0)|^2\,\d V_{0,0}\nonumber\\
		&\leq C(\alpha^4+\beta^4+\mathfrak{X}^{\frac{5}{3}}+|V(1,1)-V(0,0)|^2).
	\end{align}
	By Corollary \ref{estihighderiPQ}, it follows that
	\begin{align}\label{equ:Qesti22}
		|Q(s,t)|\leq |Q(0,0)|+C([\zeta']_0^2+(1+|\vec x|^2)\alpha^2)+C\beta ([\zeta']_0+(1+|\vec x|^2)\alpha).
	\end{align}
	Combining \eqref{equ:Qesti11} and \eqref{equ:Qesti22}, we obtain
	\begin{align*}
		\int_{\bar M}(1+|\vec x|^2)Q^2(s,t)\,\d V_{s,t}&\leq \int_{\bar M}(1+|\vec x|^2)Q^2(0,0)\,\d V_{0,0}+C(\alpha^4+\alpha^2 \beta ^2+\beta^2\mathfrak{X}^{\frac{5}{3}}+\delta \alpha^3+\delta\mathfrak{X}^{\frac{5}{3}})\nonumber\\
		&\leq C(\delta\alpha^3+\beta^4+\mathfrak{X}^{\frac{5}{3}}+|V(1,1)-V(0,0)|^2).
	\end{align*}
	Here in the first inequality, we have used $[\zeta]_0\leq C(\delta+\beta)\leq C\delta$ and
	\begin{align*}
		\int_{\bar M}(1+|\vec x|^2)[\zeta']_0^4\,\d V_{s,t}&\leq C\int_{\bar M}(1+|\vec x|^2)[\zeta]_0^4\,\d V_{s,t}+C \alpha^4\nonumber\\
		&\leq C\delta \int_{\bar M}(1+|\vec x|^2)[\zeta]_0^3\,\d V_{s,t}+C \alpha^4\\
		&\leq C\delta \int_{\bar M}(1+|\vec x|^2)([u]_0^3+[\zeta']_0^2([\zeta]_0+[u]_0))\,\d V_{s,t}+C \alpha^4\\
		&\leq C\delta(\alpha^3+\delta \mathfrak{X}^{\frac{5}{3}}+\alpha \mathfrak{X}^{\frac{5}{3}}).
	\end{align*}
	 Similarly, we can obtain the estimate for $P(s,t)$.
\end{proof}

The following two lemmas provide estimates for second derivatives of $\Phi$ and $w$.
\begin{lem}\label{lem:2ndderi3}
	For any $s,t\in [0,1]$, the following holds\textup{:}
	\begin{enumerate}[label=\textnormal{(\roman{*})}]
		\item $|\Phi_s|\leq C\big([\zeta]_2+[\na q]_1+(1+|\vec x|)[\zeta]_1\big),\, |\Phi_t|\leq C(1+|\vec x|)\beta$. In particular, $|\Phi|\leq C\big([\zeta]_2+[\na q]_1+(1+|\vec x|)([\zeta]_1+\beta)\big)$.
		\item $|\Phi_{ss}|\leq C\big([\zeta]_1^2+[\zeta]_0[\zeta]_2+(1+|\vec x|)[\zeta]_0[\zeta]_1+[\zeta]_1[\na q]_0\big)$.
		\item $|\Phi_{st}|\leq C\beta\big((1+|\vec x|)[\zeta]_1+[\zeta]_2+[\na q]_0\big)$.
		\item $|\Phi_{tt}|\leq C\beta^2(1+|\vec x|)$.
	\end{enumerate}
\end{lem}
\begin{proof}
	Item (i) on $\Phi_s$ follows from \cite[Lemma 4.7]{fang2025strong}. Note that if we replace $\zeta$ by $h_1+th_2$ and $q$ by $f_1+tf_2$, we can obtain the estimate of $\Phi_t$ by the same calculations, i.e. 
	$$|\Phi_t|\leq C\big([h_1+th_2]_2+[\na (f_1+tf_2)]_1+(1+|\vec x|)[h_1+th_2]_1\big)\leq C\beta(1+|\vec x|),$$
	where in the last inequality, we have used \eqref{equ:basicesti1}. Similarly, item (ii) follows from \cite[Lemma 4.7]{fang2025strong} and item (iv) follows from the same calculations. Now we focus on item (iii).
	Recall the following local expression of $\Ric$ and $\na^2 f$,
	\begin{align*}
		\Ric_{ij}=\partial_l\Gamma^l_{ij}-\partial_i\Gamma^l_{ij}+\Gamma^a_{ij}\Gamma^{l}_{la}-\Gamma^a_{lj}\Gamma^l_{ia},\quad f_{ij}=\partial_i\partial_j f-\Gamma^k_{ij}\partial_k f.
	\end{align*}
	
	Since $\Gamma^k_{ij}=\frac{1}{2}g^{kl}(\partial_ig_{jl}+\partial_jg_{il}-\partial_lg_{ij})$, by the same calculations as \cite[Lemma 4.7]{fang2025strong}, we have that for $s,t\in [0,1]$, 
	\begin{align*}
		|(\Gamma^k_{ij})_s| \leq& C[\zeta]_1, \quad |(\Gamma^k_{ij})_t |\leq C\beta, \quad |(\Gamma^k_{ij})_{st}|\leq C\beta [\zeta]_1,
	\end{align*}
	and thus $|(\Ric_{ij})_{st}|\leq C\beta [\zeta]_2$. And similarly, 
	\begin{align*}
		|(f_{ij})_s|\leq C\big( [\na q]_1+(1+|\vec x|)[\zeta]_1\big),\quad |(f_{ij})_t|\leq C\beta(1+|\vec x|), \quad |(f_{ij})_{st}|\leq C\beta\big((1+|\vec x|)[\zeta]_1+[\na q]_0\big).
	\end{align*}
Thus we have $|\Phi_{st}|\leq C\beta\big((1+|\vec x|)[\zeta]_1+[\zeta]_2+[\na q]_0\big)$. This completes the proof.
\end{proof}

\begin{lem}\label{lem:2ndderi2}
	For any $s,t\in [0,1]$, the following estimates hold\textup{:}
	\begin{enumerate}[label=\textnormal{(\roman{*})}]
		\item $|w_{ss}|\leq C\big( [\zeta]_2^2+[\zeta]_1[\na q]_1+[\na q]_0^2+(1+|\vec x|)([\zeta]_1^2+[\zeta]_0[\na q]_0)+(1+|\vec x|^2)[\zeta]_0^2\big)$.
		\item $|w_{tt}|\leq C(1+|\vec x|^2)\beta^2$.
		\item $|w_{st}|\leq C\beta\big( [\zeta]_2+(1+|\vec x|)([\zeta]_1+[\na q]_0+\beta)+(1+|\vec x|^2)[\zeta]_0\big). $
	\end{enumerate}
\end{lem}
\begin{proof}
	Item (i) follows from \cite[Lemma 4.10(i)]{fang2025strong}. For item (ii), by the same calculations with $\zeta$ replaced by $h_1+th_2$ and $q$ replaced by $f_1+tf_2$, we obtain
	\begin{align*}
		|w_{tt}|\leq& C\lc [h_1+th_2]_2^2+[h_1+th_2]_1[\na(f_1+tf_2)]_1+[\na (f_1+tf_2)]_0^2\rc\\
		&+C\lc(1+|\vec x|)([h_1+th_2]_1^2+[h_1+th_2]_0[\na (f_1+tf_2)]_0)+(1+|\vec x|^2)[h_1+th_2]_0^2\rc\\
		\leq & C(1+|\vec x|^2)\beta^2,
	\end{align*}
	where in the last inequality, we have used \eqref{equ:basicesti1}. The estimate for $w_{st}$ follows from similar calculations and we omit the details here.
\end{proof}

\begin{prop}\label{2ndderiat0}
	At $s=t=0$, we have\textup{:}
	\begin{enumerate}[label=\textnormal{(\roman{*})}]	
		\item $\Phi_{tt}\in \isd_N, \Phi_{ss}^0\perp \isd_N, \Phi^0_{st}\perp \isd$ and for $H_1=\Tr_{\bar g}(h_1)$, we have
		\begin{align}\label{phi2ndderi1}
			&\Phi_{ss}^0=-|\na u|^2g_N-nu\na^2 u-\frac{n}{2}\d u\otimes \d u,\\
			&\label{phi2ndderi2} \Phi_{st}^0=\frac{1}{2}uh_1-\frac{1}{2}\na^2(H_1u)+\frac{1}{4}(\d u\otimes \d H_1+\d H_1\otimes \d u).
		\end{align}
		\item $\rVert\Phi_{tt}\rVert_{L^2}\leq C\mathfrak{X}^{\frac{5}{6}}$.
		\item $\left|\int_{\bar M}(\Phi_{st})_{ij}\zeta_{jl}(h_1)_{il}\,\d V_{0,0}\right|\leq C(\alpha^4+\beta^4+\mathfrak{X}^{\frac{10}{3}})$.
		\item $\rVert \la \Phi_{ss},\zeta\ra\rVert_{L^1}\leq C(\alpha^3+\delta \mathfrak{X}^{\frac{5}{3}})$.
		\item $\rVert \la \Phi_{st},\zeta\ra\rVert_{L^1}\leq C(\alpha^4+\beta^4+\mathfrak{X}^{\frac{5}{3}})$.
		\item $\rVert \la \Phi_{tt},\zeta\ra\rVert_{L^1}\leq C\beta^4$.
		\item $\rVert \la \Phi_{tt},h_1\ra\rVert_{L^1}\leq C(\beta^4+\mathfrak{X}^{\frac{4}{3}})$.
	\end{enumerate}
\end{prop}
\begin{proof}
In this proof, all quantities are evaluated at $s=t=0$.
	Item (i) follows from \cite[Proposition 4.6]{li2023rigidity} and item (ii) follows from \cite[Equation (4.37)]{li2023rigidity} (note that the quantity $\mathfrak{X}$ defined there is different from the quantity here). For item (iii), we first claim that 
	\begin{align}\label{phi2ndesti1}
		\left|\int_{\bar M}(\Phi_{st}^0)_{ij}\zeta_{jl}(h_1)_{il}\,\d V_{0,0}\right|\leq C(\alpha^4+\beta^4+\mathfrak{X}^\frac{10}{3}).
	\end{align}
	In fact, by \eqref{phi2ndderi2}, we have
	\begin{align*}
		\left|\int_{\bar M}(\Phi_{st}^0)_{ij}(ug_N)_{jl}(h_1)_{il}\,\d V_{0,0}\right|\leq C\alpha^2\beta^2\leq C(\alpha^4+\beta^4).
	\end{align*}
	By \eqref{basicesti2}, we obtain
	\begin{align*}
		\left|\int_{\bar M}(\Phi_{st}^0)_{ij}\zeta'_{jl}(h_1)_{il}\,\d V_{0,0}\right|\leq
		C\alpha\beta^2\int_{\bar M}(1+|\vec x|^2)|\zeta'|\, \d V_{0,0}\leq C\alpha\beta^2\mathfrak{X}^{\frac{5}{6}}\leq C(\alpha^4+\beta^4+\mathfrak{X}^{\frac{10}{3}}).
	\end{align*}
	Then \eqref{phi2ndesti1} follows from combination of the above two inequalities. On the other hand, by \cite[Lemma 4.8]{li2023rigidity}, 
	$$[\Phi_{st}-\Phi_{st}^0]_0\leq C\beta ([\zeta']_2+(1+|\vec x|)[\zeta']_1+[\na q']_0),$$
	which implies that
	\begin{align}\label{phi2ndesti2}
		\left|\int_{\bar M}(\Phi_{st}-\Phi_{st}^0)_{ij}\zeta_{jl}(h_1)_{il}\,\d V_{0,0}\right|&\leq C\beta^2 \int_{\bar M}\lc [\zeta']_2+(1+|\vec x|)[\zeta']_1+[\na q']_0 \rc [\zeta]_0\, \d V_{0,0} \nonumber\\&\leq C\beta^2 \mathfrak{X}^{\frac{5}{6}}(\alpha+\mathfrak{X}^{\frac{5}{6}})\leq C(\alpha^4+\beta^4+\mathfrak{X}^{\frac{10}{3}}).
	\end{align}
	Combining \eqref{phi2ndesti1} with \eqref{phi2ndesti2}, it follows that
	\begin{align*}
		\left|\int_{\bar M}(\Phi_{st})_{ij}\zeta_{jl}(h_1)_{il}\,\d V_{0,0}\right|\leq C(\alpha^4+\beta^4+\mathfrak{X}^{\frac{10}{3}}),
	\end{align*}
	which proves item (iii). By \eqref{phi2ndderi1}, we can get that $[\Phi_{ss}^0]_0\leq C(1+|\vec x|^2)\alpha^2$, which gives
	\begin{align}\label{phi2ndesti3}
		\left|\int_{\bar M}\la \Phi_{ss}^0,\zeta\ra \, \d V_{0,0}\right|&\leq C\alpha^2\int_{\bar M}(1+|\vec x|^2)([\zeta']_0+[u]_0)\, \d V_{0,0}\nonumber\\
		&\leq C\alpha^2(\mathfrak{X}^{\frac{5}{6}}+\alpha)\leq C(\alpha^3+\mathfrak{X}^{\frac{5}{2}}).
	\end{align}
	By \cite[Lemma 4.8]{li2023rigidity}, we have
	$$[\Phi_{ss}-\Phi_{ss}^0]_0\leq C\left([\zeta']_2+(1+|\vec x|)[\zeta']_1+[\na q']_0\right)^2.$$
	Thus
	\begin{align}\label{phi2ndesti4}
		\left|\int_{\bar M}\la \Phi_{ss}-\Phi_{ss}^0,\zeta\ra \, \d V_{0,0}\right|&\leq C\int_{\bar M}\left([\zeta']_2+(1+|\vec x|)[\zeta']_1+[\na q']_0\right)^2[\zeta]_0 \, \d V_{0,0}\nonumber\\
		&\leq C\delta \mathfrak{X}^{\frac{5}{3}},
	\end{align}
	where in the second inequality, we have used the fact $[\zeta]_0\leq C\delta$.
	Then item (iv) follows from \eqref{phi2ndesti3} and \eqref{phi2ndesti4}. For item (v), we first note that by item (i), we have
	\begin{align*}
		\int_{\bar M}\la\Phi_{st}^0,ug_N\ra \,\d V_{0,0}=0.
	\end{align*}
	By \eqref{phi2ndderi2}, $[\Phi_{st}^0]_0\leq C(1+|\vec x|^2)\alpha\beta$, therefore from item (i) we have
	\begin{align}\label{phi2ndesti5}
		\left|\int_{\bar M}\la\Phi_{st}^0,\zeta\ra \,\d V_{0,0}\right|&=\left|\int_{\bar M}\la\Phi_{st}^0,\zeta'\ra \,\d V_{0,0}\right|\leq C\alpha\beta \int_{\bar M}(1+|\vec x|^2)[\zeta']_0\, \d V_{0,0}\nonumber\\
		&\leq C\alpha\beta \mathfrak{X}^{\frac{5}{6}}\leq C(\alpha^4+\beta^4+\mathfrak{X}^{\frac{5}{3}}).
	\end{align}
	On the other hand, by \cite[Lemma 4.8]{li2023rigidity}, we obtain
	\begin{align}\label{phi2ndesti6}
		\left|\int_{\bar M}\la\Phi_{st}-\Phi_{st}^0,\zeta\ra \,\d V_{0,0}\right|&\leq C\beta\int_{\bar M}([\zeta']_2+(1+|\vec x|)[\zeta']_1+[\na q']_0)([\zeta']_0+(1+|\vec x|^2)\alpha)\, \d V_{0,0}\nonumber\\
		&\leq C\beta\mathfrak{X}^{\frac{5}{3}}+C\alpha\beta\mathfrak{X}^{\frac{5}{6}} \leq C(\alpha^4+\beta^4+\mathfrak{X}^{\frac{5}{3}}).
	\end{align}
	Combining \eqref{phi2ndesti5} and \eqref{phi2ndesti6}, we obtain item (v).
	Since $\pi_{\isd}(h)=h_1+ug_N$, we have $\pi_{\isd}(\frac{1}{2}h_2+\zeta')=0$, which, by item (i), implies
	\begin{align*}
		\left|\int_{\bar M}\la \Phi_{tt},\zeta\ra\, \d V_{0,0}\right|&=\left|\int_{\bar M}\la \Phi_{tt},\zeta'\ra\, \d V_{0,0}\right|=\left|\int_{\bar M}\la \Phi_{tt},-\frac{1}{2}h_2\ra \,\d V_{0,0}\right|\leq C\beta^4.
	\end{align*}
	Here in the last inequality, we have used Lemma \ref{lem:2ndderi3} (iv). This gives item (vi). Finally, item (vii) follows directly from item (ii) and Cauchy-Schwarz inequality. This completes the proof.
\end{proof}

Next, we estimate the third derivatives of $\Phi$:
\begin{lem}\label{3rdphiesti}
	For any $s,t\in [0,1]$, the following holds\textup{:}
	\begin{enumerate}[label=\textnormal{(\roman{*})}]
		\item $|\Phi_{sss}|\leq C\big([\zeta]_2^3+(1+|\vec x|)[\zeta]_1^3+[\zeta]_1^2[\na q]_0\big)$.
		\item $|\Phi_{sst}|\leq C\beta\big([\zeta]_2^2+(1+|\vec x|)[\zeta]_1^2+[\zeta]_1[\na q]_0\big)$.
		\item $|\Phi_{stt}|\leq C\beta^2\big([\zeta]_2+(1+|\vec x|)[\zeta]_1+[\na q]_0\big)$.
		\item $|\Phi_{ttt}|\leq C\beta^3(1+|\vec x|)$.
	\end{enumerate}
\end{lem}
\begin{proof}
	By a similar argument as in proof of Lemma \ref{lem:2ndderi3} (see also \cite[Lemma 4.7]{li2023rigidity}), we have
	\begin{align*}
		[\Ric_{sss}]_0\leq C[\zeta]_2^3,\quad [\Ric_{sst}]_0\leq C\beta[\zeta]_2^2,\quad [\Ric_{stt}]_0\leq C\beta^2[\zeta]_2,\quad [\Ric_{ttt}]_0\leq C\beta^3,
	\end{align*}
	and 
	\begin{align*}
		&[(\na^2f)_{sss}]_0\leq C\big((1+|\vec x|)[\zeta]_1^3+[\zeta]_1^2[\na q]_0\big),\\ 
		&[(\na^2 f)_{sst}]_0\leq C\beta\big((1+|\vec x|)[\zeta]_1^2+[\zeta]_1[\na q]_0\big),\\
	    &[(\na^2 f)_{stt}]_0\leq C\beta^2\big((1+|\vec x|)[\zeta]_1+[\na q]_0\big),\\ 
	    &[(\na^2 f)_{ttt}]_0\leq C\beta^3(1+|\vec x|).
	\end{align*}
	Combining all the above inequalities, the lemma follows immediately.
\end{proof}

Similarly, the following estimates for third derivatives of $w$ hold.
\begin{lem}
	\label{lem:3rdw}For any $s,t\in [0,1]$, we have\textup{:}
	\begin{enumerate}[label=\textnormal{(\roman{*})}]
		\item $|w_{sss}|\leq C( [\zeta]_2^3+[\zeta]_1^2[\na q]_1+[\zeta]_0[\na q]_0^2 +(1+|\vec x|)([\zeta]_1^3+[\zeta]_0^2[\na q]_0)+(1+|\vec x|^2)[\zeta]_0^3)$.
		\item $|w_{sst}|\leq C\beta \big( [\zeta]_2^2+[\zeta]_1[\na q]_1+[\na q]_0^2+(1+|\vec x|)([\zeta]_1^2+[\zeta]_0[\na q]_0)+(1+|\vec x|^2)[\zeta]_0^2\big)$.
		\item $|w_{stt}|\leq C\beta^2 \big( [\zeta]_2+(1+|\vec x|)([\zeta]_1+[\na q]_0+\beta)+(1+|\vec x|^2)[\zeta]_0\big)$.
		\item $|w_{ttt}| \leq C\beta^3(1+|\vec x|^2)$.
	\end{enumerate}
\end{lem}

By the formulas of third derivatives of $V$ (see \eqref{equ:3rdQderi1}, \eqref{equ:3rdQderi2}, \eqref{equ:3rdQderi3} and \eqref{equ:3rdQderi4}) and Corollary \ref{cor:1stderiat0}, we know that at $s=t=0$, we have
		\begin{align*}
		V_{sss}(0,0)&=\int_{\bar M} \zeta_{ij}\zeta_{jl}\zeta_{il}-\frac{3}{2}|\zeta|^2Q(0,0)+Q^3(0,0)\,\d V_{0,0},\\		
			V_{ttt}(0,0)&=\int_{\bar M}-\frac{3}{2}\la h_1,h_2\ra+(h_1)_{ij}(h_1)_{jl}(h_1)_{il}\,\d V_{0,0},\\	
		V_{sst}(0,0)&=\int_{\bar M}\zeta_{ij}\zeta_{jl}(h_1)_{il}-\la \zeta,h_1\ra Q(0,0)\,\d V_{0,0},\\
		 V_{stt}(0,0)&=\int_{\bar M}-\frac{1}{2}\la h_2,\zeta\ra+(h_1)_{ij}\zeta_{jl}(h_1)_{il}+\left(\frac{1}{2}\Tr_{\bar g}h_2-f_2-\frac{1}{2}|h_1|^2\right)Q(0,0)\,\d V_{0,0}.
	\end{align*}

\begin{lem}\label{4thderiV}
	The following estimates hold\textup{:}
	\begin{enumerate}[label=\textnormal{(\roman{*})}]
		\item $\sup_{s,t\in [0,1]}|V_{sss}|\leq C (\alpha^3+\beta^4+\mathfrak{X}^{\frac{5}{3}}+|V(1,1)-V(0,0)|^2)$.
		\item $\sup_{s,t\in [0,1]}\big(|V_{ssss}|+|V_{ssst}|+|V_{sstt}|+|V_{sttt}|+|V_{tttt}|\big)\leq C(\alpha^4+\beta^4+\mathfrak{X}^{\frac{5}{3}}+|V(1,1)-V(0,0)|^2).$
		\end{enumerate}
\end{lem}
\begin{proof}
	Recall from \eqref{equ:3rdQderi1} that
	$$V_{sss}=\int_{\bar M}Q_{ss}+3QQ_s+Q^3\,\d V_{s,t}.$$	
	By Corollary \ref{estihighderiPQ}, we have $[Q_{ss}]_0\leq C\big([\zeta']_0^3+(1+|\vec x|^6)\alpha^3\big)$, which implies
	\begin{align}\label{estiV3rd1}
		\left|\int_{\bar M}Q_{ss}\, \d V_{s,t}\right|&\leq C\int_{\bar M}[\zeta']_0^3+(1+|\vec x|^6)\alpha^3\, \d V_{s,t}\nonumber\\
		&\leq C\int_{\bar M} ([\zeta]_0+[u]_0)[\zeta']_0^2 +(1+|\vec x|^6)\alpha^3\, \d V_{s,t} \nonumber\\
		&\leq C(\delta\mathfrak{X}^{\frac{5}{3}}+\alpha \mathfrak{X}^{\frac{5}{3}})+C\alpha^3\nonumber\\
		&\leq C(\alpha^3+\delta\mathfrak{X}^{\frac{5}{3}}).
	\end{align}
	Since $[Q_s]_0\leq C\big([\zeta']_0^2+(1+|\vec x|^4)\alpha^2\big)$, we have by Cauchy-Schwarz inequality and Proposition \ref{pro:Qesti2} that
	\begin{align}\label{estiV3rd2}
		\left|\int_{\bar M}QQ_s\, \d V_{s,t}\right|&\leq C\int_{\bar M}\big([\zeta']_0^2+(1+|\vec x|^4)\alpha^2\big)|Q|\,\d V_{s,t}\nonumber\\
		&\leq C\delta \mathfrak{X}^{\frac{5}{3}}+C\alpha^2(\alpha^2+\beta^2+\mathfrak{X}^{\frac{5}{6}}+|V(1,1)-V(0,0)|)\nonumber\\
		&\leq C(\alpha^4+\beta^4+\mathfrak{X}^{\frac{5}{3}}+|V(1,1)-V(0,0)|^2).
	\end{align}
	According to Proposition \ref{pro:Qesti2}, it follows that
	\begin{align}\label{estiV3rd3}
		\left|\int_{\bar M}Q^3\, \d V_{s,t}\right|\leq C\delta \left|\int_{\bar M}Q^2\, \d V_{s,t}\right|\leq C\delta (\alpha^4+\beta^4+\mathfrak{X}^{\frac{5}{3}}+|V(1,1)-V(0,0)|^2).	\end{align}
	Then item (i) follows from combination of \eqref{estiV3rd1}, \eqref{estiV3rd2} and \eqref{estiV3rd3}.
	
	Now we focus on the estimates of fourth derivatives of $V$. We only give all details of $\sup_{s,t\in [0,1]}|V_{sstt}|\leq C(\alpha^4+\beta^4+\mathfrak{X}^{\frac{5}{3}}+|V(1,1)-V(0,0)|^2)$ and other estimates will follow in a similar way. Recall that $V_{stt}=\int_{\bar M}Q_{tt}+2Q_tP+QP_t+QP^2\,\d V_{s,t}$, thus
	\begin{align}\label{4thVesti1}
		V_{sstt}=\int_{\bar M}Q_{stt}+2Q_{st}P+2Q_tP_s+Q_sP_t+QP_{st}+Q_sP^2+2QPP_s+(Q_{tt}+2Q_tP+QP_t+QP^2)Q\,\d V_{s,t}.
	\end{align}
	By Propositions \ref{pro:Pesti}, \ref{pro:Qesti2} and the fact that all derivatives of $P,Q$ up to order 2 are bounded by $C\delta$, we have
	\begin{align}\label{4thVesti2}
		\left|\int_{\bar M}Q_sP^2+2QPP_s+2Q_tPQ+Q^2P_t+Q^2P^2\,\d V_{s,t}\right|\leq C(\alpha^4+\beta^4+\mathfrak{X}^{\frac{5}{3}}+|V(1,1)-V(0,0)|^2).
	\end{align}
    Since $|Q_{stt}|\leq |\zeta*\zeta*h_2|+|\zeta*\zeta*(h_1+th_2)*(h_1+th_2)|$, it follows 
    \begin{align}\label{4thVesti3}
    	\left|\int_{\bar M}Q_{stt}\,\d V_{s,t}\right|&\leq C\beta^2\int_{\bar M}[\zeta]_0^2\,\d V_{s,t} \leq C(\alpha^4+\beta^4+\mathfrak{X}^{\frac{10}{3}}).
    \end{align}
    By Corollary \ref{estihighderiPQ}, we obtain
    \begin{align*}
    	\rVert P_{st}\rVert_{L^2}+\rVert Q_{st}\rVert_{L^2}+\rVert Q_{tt}\rVert_{L^2}&\leq C\beta^2(\alpha+\mathfrak{X}^{\frac{5}{6}})+C\beta(\alpha^2+\mathfrak{X}^{\frac{5}{3}}) \leq C(\alpha^2\beta+\alpha\beta^2+\beta^2\mathfrak{X}^{\frac{5}{6}}+\beta\mathfrak{X}^{\frac{5}{3}}),
    \end{align*}
    which implies that
    \begin{align}\label{4thVesti4}
    	\left|\int_{\bar M}2Q_{st}P+QP_{st}+Q_{tt}Q\,\d V_{s,t}\right|&\leq C(\rVert P_{st}\rVert_{L^2}+\rVert Q_{st}\rVert_{L^2}+\rVert Q_{tt}\rVert_{L^2})(\rVert P\rVert_{L^2}+\rVert Q\rVert_{L^2})\nonumber\\
    	&\leq C(\alpha^2\beta+\alpha\beta^2+\beta^2\mathfrak{X}^{\frac{5}{6}}+\beta\mathfrak{X}^{\frac{5}{3}})(\alpha^2+\beta^2+\mathfrak{X}^{\frac{5}{6}}+|V(1,1)-V(0,0)|)\nonumber\\
    	&\leq C(\alpha^4+\beta^4+\mathfrak{X}^{\frac{5}{3}}+|V(1,1)-V(0,0)|^2).
    \end{align}
    By Corollary \ref{estihighderiPQ}, we get $|P_s|+|Q_t|\leq C\beta[\zeta]_0, |Q_s|\leq C[\zeta]_0^2$ and $|P_t|\leq C\beta^2$, which implies
    \begin{align}\label{4thVesti5}
    	\left|\int_{\bar M}2Q_tP_s+Q_sP_t\,\d V_{s,t}\right|&\leq C\beta^2\int_{\bar M}[\zeta]_0^2\,\d V_{s,t}\leq C\beta^2\big(\alpha^2+\mathfrak{X}^{\frac{5}{3}}\big)\leq C(\alpha^4+\beta^4+\mathfrak{X}^{\frac{10}{3}}).
    \end{align}
    Plugging \eqref{4thVesti2}, \eqref{4thVesti3}, \eqref{4thVesti4} and \eqref{4thVesti5} into \eqref{4thVesti1}, we obtain 
    $$\sup_{s,t\in [0,1]}|V_{sstt}|\leq C(\alpha^4+\beta^4+\mathfrak{X}^{\frac{5}{3}}+|V(1,1)-V(0,0)|^2).$$
    This completes the proof.
\end{proof}

As a direct corollary of Lemma \ref{4thderiV} and Taylor expansion of $V$, we have
\begin{cor}\label{taylorofV6}
	\begin{align*}
		&\left|V_s+\frac{1}{2}\left(V_{ss}+2V_{st}+V_{tt}\right)+\frac{1}{6} \left(3V_{sst}+3V_{stt}+V_{ttt}\right)\right|
		\leq C \big(\alpha^3+\beta^4+\mathfrak{X}^{\frac{5}{3}}+|V(1,1)-V(0,0)|\big),
	\end{align*}
	where on the left hand side, all derivatives are evaluated at $(0,0)$.
\end{cor}

Now we calculate second derivatives of $W$ at $s=t=0$. By Lemma \ref{lem:1stderi}, for any $s,t\in [0,1]$, we have
\begin{align}\label{2ndderiW1}
 W_{ss}=\int_{\bar M} -2\Phi_{ij}\zeta_{jl}\zeta_{il}+\la \Phi_s,\zeta\ra +\la \Phi,\zeta\ra Q+Q_s(w-1)+Qw_s+Q^2(w-1)\, \d V_{s,t},
\end{align}
and 
\begin{align}\label{2ndederiW2}
	W_{st}=\int_{\bar M} -2\Phi_{ij}\zeta_{jl}(h_1+th_2)_{il}+\la \Phi_t,\zeta\ra +\la \Phi,\zeta\ra P+Q_t(w-1)+Qw_t+Q(w-1)P\, \d V_{s,t},
\end{align}
and
\begin{align}\label{2ndderiW3}
	W_{tt}=\int_{\bar M}& -2\Phi_{ij}(h_1+th_2)_{jl}(h_1+th_2)_{il}+\la \Phi,h_2\ra +\la \Phi_t,h_1+th_2\ra +\la \Phi,h_1+th_2\ra P\nonumber\\
	&+P_t(w-1)+Pw_t+P^2(w-1)\,\d V_{s,t}.
\end{align}

Therefore, when evaluating at $s=t=0$, we obtain from \eqref{phideri1} and Corollary \ref{cor:1stderiat0} that
\begin{align}\label{2ndderiWat01}
	W_{ss}=&\int_{\bar M}\la\Phi_s,\zeta\ra +Q_s(w-1)+Qw_s+Q^2(w-1)\,\d V_{0,0}\nonumber\\
	=&\frac{1}{2}\int_{\bar M}\la \LL\zeta+2\Div_{\bar f}^*\,\Div_{\bar f}\zeta+2\na^2 Q(0,0),\zeta\ra \,\d V_{0,0}+\int_{\bar M}2|\na Q(0,0)|^2-Q^2(0,0)+Q(0,0)\Div_{\bar f}\,\Div_{\bar f}\zeta\,\d V_{0,0}\nonumber\\
		&+\big(\bar W-1\big)V_{ss}(0,0),
\end{align}
and (recall that $\Phi_t=0,\, \Div_{\bar f}h_1=0,\, P=0$ at $s=t=0$)
\begin{align}\label{2ndderiWat02}
	W_{st}&=\int_{\bar M}Q_t(w-1)+Qw_t\,\d V_{0,0} =(\bar{W}-1)V_{st}(0,0),
\end{align}
and
\begin{align}\label{2ndderiWat03}
	W_{tt}&=\int_{\bar M} P_t(w-1) \,\d V_{0,0} =(\bar W-1)V_{tt}(0,0).
\end{align}

Then we calculate third derivatives of $W$. From \eqref{2ndderiW1}, we can calculate $W_{sss}$ as
\begin{align}\label{equ:3rdWsss}
	W_{sss}=&\int_{\bar M}-2(\Phi_s)_{ij}\zeta_{jl}\zeta_{il}+6\Phi_{ij}\zeta_{jl}\zeta_{lm}\zeta_{im}+\la \Phi_{ss},\zeta\ra -2(\Phi_s)_{ij}\zeta_{jl}\zeta_{il}\,\d V_{s,t}\nonumber\\
	&+\int_{\bar M}\la\Phi_s,\zeta\ra Q-2\Phi_{ij}\zeta_{jl}\zeta_{il}Q+\la \Phi,\zeta\ra Q_s+Q_{ss}(w-1)+2Q_sw_s+Qw_{ss}+2QQ_s(w-1)+Q^2w_s\,\d V_{s,t}\nonumber\\
	&+\int_{\bar M}\big(-2\Phi_{ij}\zeta_{jl}\zeta_{il}+\la \Phi_s,\zeta\ra +\la \Phi,\zeta\ra Q+Q_s(w-1)+Qw_s+Q^2(w-1)\big)Q\,\d V_{s,t}\nonumber\\
	=&\int_{\bar M}-4(\Phi_s)_{ij}\zeta_{jl}\zeta_{il}+6\Phi_{ij}\zeta_{jl}\zeta_{lm}\zeta_{im}+\la \Phi_{ss},\zeta\ra \,\d V_{s,t}\nonumber\\
	&+\int_{\bar M}2\la\Phi_s,\zeta\ra Q-4\Phi_{ij}\zeta_{jl}\zeta_{il}Q+\la \Phi,\zeta\ra Q_s+Q_{ss}(w-1)+2Q_sw_s+Qw_{ss}+3QQ_s(w-1)+2Q^2w_s\,\d V_{s,t}\nonumber\\
	&+\int_{\bar M}\big(\la \Phi,\zeta\ra Q+Q^2(w-1)\big)Q\,\d V_{s,t}.
\end{align}
In particular, we have
\begin{align*}
W_{sss}(0,0)=\int_{\bar M}-4(\Phi_s)_{ij}\zeta_{jl}\zeta_{il}+\langle \Phi_{ss},\zeta\rangle +2\langle\Phi_s,\zeta\rangle Q+2Q_s w_s+Qw_{ss}+2Q^2 w_s\, \d V_{0,0}+(\bar W-1)V_{sss}(0,0).	
\end{align*}

Similarly, we calculate $W_{sst}$ from \eqref{2ndderiW1} as
\begin{align*}
	W_{sst}=&\int_{\bar M}-2(\Phi_t)_{ij}\zeta_{jl}\zeta_{il}+6\Phi_{ij}\zeta_{jl}\zeta_{lm}(h_1+th_2)_{im}+\la \Phi_{st},\zeta\ra-2(\Phi_s)_{ij}\zeta_{jl}(h_1+th_2)_{il}\,\d V_{s,t}\nonumber\\
	+&\int_{\bar M}\la\Phi_t,\zeta\ra Q-2\Phi_{ij}\zeta_{jl}(h_1+th_2)_{il}Q+\la \Phi,\zeta\ra Q_t+Q_{st}(w-1)+Q_sw_t+Q_tw_s+Qw_{st}\nonumber\\
	&+2QQ_t(w-1)+Q^2w_t\,\d V_{s,t}\nonumber\\
	+&\int_{\bar M}\big( -2\Phi_{ij}\zeta_{jl}\zeta_{il}+\la \Phi_s,\zeta\ra +\la \Phi,\zeta\ra Q+Q_s(w-1)+Qw_s+Q^2(w-1)\big)P\,\d V_{s,t}.
\end{align*}
Note that by Corollary \ref{cor:1stderiat0}, $\Phi=0,\,\Phi_t=0,\, P=0,\, w_t=0$ at $s=t=0$, then we obtain
\begin{align*}
	W_{sst}(0,0)=&\int_{\bar M}\la \Phi_{st},\zeta\ra-2(\Phi_s)_{ij}\zeta_{jl}(h_1)_{il}+Q_tw_s+Qw_{st}\,\d V_{0,0}+(\bar W-1)V_{sst}(0,0).
\end{align*}
And we calculate $W_{stt}$ from \eqref{2ndederiW2} as
\begin{align}\label{3rdderiW3}
	W_{stt}=&\int_{\bar M}-2(\Phi_t)_{ij}\zeta_{jl}(h_1+th_2)_{il}-2\Phi_{ij}\zeta_{jl}(h_2)_{il}+6\Phi_{ij}\zeta_{jl}(h_1+th_2)_{lm}(h_1+th_2)_{im}\,\d V_{s,t}\nonumber\\
	&+\int_{\bar M}\la \Phi_{tt},\zeta\ra-2(\Phi_t)_{ij}\zeta_{jl}(h_1+th_2)_{il}+\la \Phi_t,\zeta\ra P-2\Phi_{ij}\zeta_{jl}(h_1+th_2)_{il}P\,\d V_{s,t}\nonumber\\
	&+\int_{\bar M}\la \Phi,\zeta\ra P_t+ Q_{tt}(w-1)+2Q_tw_t+Qw_{tt}+Q_t(w-1)P+Qw_tP+Q(w-1)P_t\,\d V_{s,t}\nonumber\\
	&+\int_{\bar M}\big( -2\Phi_{ij}\zeta_{jl}(h_1+th_2)_{il}+\la \Phi_t,\zeta\ra +\la \Phi,\zeta\ra P+Q_t(w-1)+Qw_t+Q(w-1)P\big)P\, \d V_{s,t},
\end{align}
which, when evaluating at $s=t=0$, gives
\begin{align*}
	W_{stt}(0,0)=&\int_{\bar M}\la \Phi_{tt},\zeta\ra + Qw_{tt}\,\d V_{0,0}
	+(\bar W-1)V_{stt}(0,0).
\end{align*}

Lastly we calculate $W_{ttt}$ from \eqref{2ndderiW3} as
\begin{align*}
	W_{ttt}=&\int_{\bar M}-2(\Phi_t)_{ij}(h_1+th_2)_{jl}(h_1+th_2)_{il}-4\Phi_{ij}(h_2)_{jl}(h_1+th_2)_{il}\,\d V_{s,t}\nonumber\\
	&+\int_{\bar M}6\Phi_{ij}(h_1+th_2)_{jl}(h_1+th_2)_{lm}(h_1+th_2)_{im}+\la \Phi_t,h_2\ra -2\Phi_{ij}(h_2)_{jl}(h_1+th_2)_{il}\,\d V_{s,t}\nonumber\\
	&+\int_{\bar M}\la \Phi_{tt},h_1+th_2\ra +\la \Phi_t,h_2\ra-2(\Phi_t)_{ij}(h_1+th_2)_{jl}(h_1+th_2)_{il}+\la \Phi_t,h_1+th_2\ra P+\la \Phi,h_2\ra P\,\d V_{s,t}\nonumber\\
	&+\int_{\bar M}-2\Phi_{ij}(h_1+th_2)_{jl}(h_1+th_2)_{il}P+\la \Phi,h_1+th_2\ra P_t+P_{tt}(w-1)+2P_tw_t+Pw_{tt}+2PP_t(w-1)\,\d V_{s,t}\nonumber\\
	&+\int_{\bar M}P^2w_t+\big( -2\Phi_{ij}(h_1+th_2)_{jl}(h_1+th_2)^{il}+\la \Phi,h_2\ra +\la \Phi_t,h_1+th_2\ra +\la \Phi,h_1+th_2\ra P\nonumber\\
	&+P_t(w-1)+Pw_t+P^2(w-1)\big)P\, \d V_{s,t}.
\end{align*}
When evaluating at $s=t=0$, it follows that 
\begin{align*}
	W_{ttt}(0,0)=&\int_{\bar M}\la \Phi_{tt},h_1\ra\,\d V_{0,0}+(\bar W-1)V_{ttt}(0,0).
\end{align*}

\begin{lem}\label{3rdderiWesti}
	At $s=t=0$, we have the following estimates\textup{:}
	\begin{align*}
		\left|\int_{\bar M} Q_tw_s \,\d V_{0,0}\right|+\left|\int_{\bar M} Qw_{tt} \,\d V_{0,0}\right|+\left|\int_{\bar M} Qw_{st} \,\d V_{0,0}\right|\leq C(\alpha^4+\beta^4+\mathfrak{X}^{\frac{5}{3}}+|V(1,1)-V(0,0)|^2).
	\end{align*}
\end{lem}
\begin{proof}
	By Corollary \ref{cor:1stderiat0}, we know that at $s=t=0$, $Q_t=-\frac{1}{2}\la h_1,\zeta\ra$ and $w_s=-2(\Delta_{\bar f}+\frac{1}{2})Q(0,0)+\Div_{\bar f}\,\Div_{\bar f} \zeta$. By Proposition \ref{pro:Qesti2} and Cauchy-Schwarz inequality, we have
	\begin{align}\label{erroresti1}
	\int_{\bar M}|\la h_1,\zeta\ra||Q(0,0)|\,\d V_{0,0}&\leq C\beta \rVert \zeta\rVert_{L^2}\rVert Q(0,0)\rVert_{L^2}\nonumber\\
	&\leq C\beta(\alpha+\mathfrak{X}^{\frac{5}{6}})(\alpha^2+\beta^2+\mathfrak{X}^{\frac{5}{6}}+|V(1,1)-V(0,0)|)\nonumber\\
	&\leq C(\alpha^4+\beta^4+\mathfrak{X}^{\frac{5}{3}}+|V(1,1)-V(0,0)|^2).
	\end{align}
	By the fact $\Div_{\bar f}(ug_N)=0$, we obtain that $[\Div_{\bar f}\,\Div_{\bar f} \zeta]_0\leq C\big([\zeta']_2+(1+|\vec x|)[\zeta']_1+(1+|\vec x|^2)[\zeta']_0\big)$, which implies
	\begin{align}\label{erroresti2}
		\int_{\bar M}|\la h_1,\zeta\ra||\Div_{\bar f}\,\Div_{\bar f} \zeta|\,\d V_{0,0}&\leq C\beta \int_{\bar M}[\zeta]_0\big([\zeta']_2+(1+|\vec x|)[\zeta']_1+(1+|\vec x|^2)[\zeta']_0\big)\,\d V_{0,0}\nonumber\\
		&\leq C\beta \int_{\bar M}\lc (1+|\vec x|^2)\alpha+[\zeta']_0\rc \left([\zeta']_2+(1+|\vec x|)[\zeta']_1+(1+|\vec x|^2)[\zeta']_0\right)\,\d V_{0,0}\nonumber\\
		& \leq C\beta(\alpha\mathfrak{X}^{\frac{5}{6}}+\mathfrak{X}^{\frac{5}{3}}) \leq C(\alpha^4+\beta^4+\mathfrak{X}^{\frac{5}{3}}).
	\end{align}
	Note that $\int_{\bar M}|\na Q(0,0)|^2\,\d V_{0,0}\leq C\mathfrak{X}^{\frac{5}{3}}$ (see \eqref{taylorofV4}), then it follows that
	\begin{align}\label{erroresti3}
		\left|\int_{\bar M}\la h_1,\zeta\ra \Delta_{\bar f}Q(0,0)\,\d V_{0,0}\right|&\leq \left|\int_{\bar M}\big\la \na\la h_1,\zeta\ra, \na Q(0,0)\big\ra\,\d V_{0,0}\right| \leq C\beta (\alpha+\mathfrak{X}^{\frac{5}{6}})\mathfrak{X}^{\frac{5}{6}} \leq C(\alpha^4+\beta^4+\mathfrak{X}^{\frac{5}{3}}).
	\end{align}
	Combining \eqref{erroresti1},\,\eqref{erroresti2} and \eqref{erroresti3}, it follows that
	\begin{align}\label{equ:erroresti111}
		\left|\int_{\bar M} Q_tw_s \,\d V_{0,0}\right|\leq C(\alpha^4+\beta^4+\mathfrak{X}^{\frac{5}{3}}+|V(1,1)-V(0,0)|^2).
\end{align}
	By Proposition \ref{pro:Qesti2} and Lemma \ref{lem:2ndderi2} and Cauchy-Schwarz inequality, it follows that
	\begin{align}\label{erroresti4}
		\left|\int_{\bar M} Qw_{tt} \,\d V_{0,0}\right|\leq C\beta^2 \rVert Q\rVert_{L^2}&\leq C\beta^2(\alpha^2+\beta^2+\mathfrak{X}^{\frac{5}{6}}+|V(1,1)-V(0,0)|)\nonumber\\
		&\leq C(\alpha^4+\beta^4+\mathfrak{X}^{\frac{5}{3}}+|V(1,1)-V(0,0)|^2).
	\end{align}
	By Lemma \ref{lem:2ndderi2} again, we have $|w_{st}|\leq C\beta\big( [\zeta]_2+(1+|\vec x|)([\zeta]_1+[\na q]_0+\beta)+(1+|\vec x|^2)[\zeta]_0 \big)$, which implies that
	\begin{align}\label{erroresti5}
		\left|\int_{\bar M} Qw_{st} \,\d V_{0,0}\right|&\leq C\beta \rVert Q\rVert_{L^2}\big\rVert [\zeta]_2+(1+|\vec x|)([\zeta]_1+[\na q]_0+\beta)+(1+|\vec x|^2)[\zeta]_0 \big\rVert_{L^2}\nonumber\\
		&\leq C\beta(\alpha^2+\beta^2+\mathfrak{X}^{\frac{5}{6}}+|V(1,1)-V(0,0)|)(\alpha+\beta+\mathfrak{X}^{\frac{5}{6}})\nonumber\\
		&\leq C(\alpha^4+\beta^4+\mathfrak{X}^{\frac{5}{3}}+|V(1,1)-V(0,0)|^2).
	\end{align}
	The lemma now follows from \eqref{equ:erroresti111},\,\eqref{erroresti4} and \eqref{erroresti5}. This finishes the proof.
\end{proof}

In particular, from Lemmas \ref{lem:2ndderi3} and \ref{3rdderiWesti} and Proposition \ref{2ndderiat0}, we have
\begin{align}
\label{equ:diff3rd}\left|W_{sst}(0,0)-(\bar W-1)V_{sst}(0,0)\right|\leq C(\alpha^4+\beta^4+\alpha^2\beta+\mathfrak{X}^{\frac{5}{3}}+|V(1,1)-V(0,0)|^2),	
\end{align}
where we have used from Lemma \ref{lem:2ndderi3} (i) that
\begin{align*}
	\left|\int_{\bar M}-2(\Phi_s)_{ij}\zeta_{jl}(h_1)_{il}\,\d V_{0,0}\right|\leq C\beta(\alpha^2+\mathfrak{X}^{\frac{5}{3}}).
\end{align*}

\begin{prop}\label{4thderiW}
	There exists a constant $C$ such that
	\begin{enumerate}[label=\textnormal{(\roman{*})}]
		\item $\sup_{s,t\in [0,1]}|W_{sss}|\leq C (\alpha^3+\beta^4+\mathfrak{X}^{\frac{5}{3}}+|V(1,1)-V(0,0)|^2)$.
	\item $\sup_{s,t\in [0,1]}\big(|W_{ssst}|+|W_{sstt}|+|W_{sttt}|+|W_{tttt}|\big)\leq C (\alpha^4+\beta^4+\mathfrak{X}^{\frac{5}{3}}+|V(1,1)-V(0,0)|^2)$.
	
	\end{enumerate}

\end{prop}
\begin{proof}
Recall that we have the following rough estimates:
\begin{align}\label{equ:roughesti}
	[Q]_2\leq C\delta,\quad [\Phi]_0+[\Phi_s]_0 \leq C(1+|\vec x|)\delta,\quad [w_s]_0\leq (1+|\vec x|^2)\delta,\quad [\la \Phi,\zeta\ra]_0+[\Phi_{ss}]_0\leq C(1+|\vec x|)\delta^2.
\end{align}
By Lemma \ref{lem:2ndderi3}, we have
\begin{align}\label{3rdWesti1}
	\left|\int_{\bar M}-4(\Phi_s)_{ij}\zeta_{jl}\zeta_{il}+6\Phi_{ij}\zeta_{jl}\zeta_{lm}\zeta_{im}+\la \Phi_{ss},\zeta\ra \,\d V_{s,t}\right|\leq C(\alpha^3+\delta\mathfrak{X}^{\frac{5}{3}}),
\end{align}
where we have used
\begin{align*}
\int_{\bar M} [\zeta]_2[\zeta]_0^2\,\d V_{s,t}&\leq C\int_{\bar M} [u]_2^3+ [\zeta']_2[\zeta']_0^2+[\zeta']_2[u]_0^2\,\d V_{s,t}\\
& \leq C\int_{\bar M} [u]_2^3+(1+|\vec x|^2)\delta [\zeta']_0^2+(1+|\vec x|^2)\alpha^2 [\zeta']_2\,\d V_{s,t}\\
& \leq C(\alpha^3+\delta \mathfrak{X}^{\frac{5}{3}}+\alpha^2\mathfrak{X}^{\frac{5}{6}}),	
\end{align*}
and similarly,
\begin{align*}
\int_{\bar M}	 (1+|\vec x|)[\zeta]_1[\zeta]_0^2+[\zeta]_1^2[\zeta]_0+[\na q]_1[\zeta]_0^2+[\na q]_0[\zeta]_0[\zeta]_1\,\d V_{s,t}\leq C(\alpha^3+ \delta\mathfrak{X}^{\frac{5}{3}}).
\end{align*}
By Lemma \ref{lem:PQesti2} and \eqref{equ:roughesti}, we obtain
\begin{align}\label{3rdWesti2}
\int_{\bar M}|w_s|Q^2+|\la \Phi,\zeta\ra | Q^2\,\d V_{s,t} \leq C\delta(\delta \alpha^3+\beta^4+\mathfrak{X}^{\frac{5}{3}}+|V(1,1)-V(0,0)|^2).
\end{align}
By Lemma \ref{lem:2ndderi3}, $|\Phi_s|\leq C\big([\zeta]_2+[\na q]_1+(1+|\vec x|)[\zeta]_1\big)$, which implies
\begin{align}\label{3rdWesti3}
	\left|\int_{\bar M}\la\Phi_s,\zeta\ra Q\,\d V_{s,t}\right|&\leq C\int_{\bar M}\big([\zeta]_2+[\na q]_1+(1+|\vec x|)[\zeta]_1\big)[\zeta]_0|Q|\,\d V_{s,t}\nonumber\\
	&\leq C\delta \int_{\bar M} [\zeta']_2^2+[\na q']_1^2+(1+|\vec x|)
[\zeta']_1^2\,\d V_{s,t}+C\alpha^2\int_{\bar M}(1+|\vec x|^4)|Q|\,\d V_{s,t}\nonumber\\
&\leq C\delta\mathfrak{X}^{\frac{5}{3}}+C\alpha^2(\alpha^2+\beta^2+\mathfrak{X}^{\frac{5}{6}}+|V(1,1)-V(0,0)|)\nonumber\\
&\leq C(\alpha^4+\beta^4+\mathfrak{X}^{\frac{5}{3}}+|V(1,1)-V(0,0)|^2),
\end{align}
where in the second inequality, we have used the fact that $|Q|\leq C\delta$ and in the third inequality, we have used Proposition \ref{pro:Qesti2} and Cauchy-Schwarz inequality. By a similar argument, we also have
\begin{align}\label{3rdWesti4}
	\left|\int_{\bar M}\Phi_{ij}\zeta_{jl}\zeta_{il} Q\,\d V_{s,t}\right|\leq C(\alpha^4+\beta^4+\mathfrak{X}^{\frac{5}{3}}+|V(1,1)-V(0,0)|^2).
\end{align}
By Lemma \ref{lem:1stderi}, $Q_s=-\frac{1}{2}|\zeta|^2$, and thus from \eqref{equ:roughesti}, we get
\begin{align}\label{3rdWesti5}
	\left|\int_{\bar M}\la \Phi,\zeta\ra Q_s\,\d V_{s,t}\right|\leq C\delta \int_{\bar M}(1+|\vec x|)|\zeta'|^3+\alpha^3(1+|\vec x|^7)\,\d V_{s,t}\leq C(\alpha^3+\delta\mathfrak{X}^{\frac{5}{3}}).
\end{align}
Since $w_s=-2(\Delta_{ f}+\frac{1}{2})Q+\la\Phi,\zeta\ra+\Div_{f}\,\Div_{\bar f}\zeta$, we obtain
\begin{align}\label{3rdWesti6}
	\left|\int_{\bar M}Q_sw_s\,\d V_{s,t}\right|&\leq C\int_{\bar M}|\zeta|^2|\Delta_{ f}Q|+|Q||\zeta|^2+|\zeta|^2|\Div_{f}\,\Div_{\bar f}\zeta|\,\d V_{s,t} +C\left|\int_{\bar M}\la \Phi,\zeta\ra Q_s\,\d V_{s,t}\right|\nonumber\\
	&\leq C\int_{\bar M}[\zeta]_0^2[\na Q]_1+(1+|\vec x|^2)[\zeta]_0^2[\na Q]_0+[Q]_0^2+[\zeta]_0^4+[\zeta]_0^2[\Div_{f}\, \Div_{\bar f}\zeta]_0\,\d V_{s,t}+C(\alpha^3+\delta\mathfrak{X}^{\frac{5}{3}})\nonumber\\
	&\leq C(\alpha^3+\beta^4+\mathfrak{X}^{\frac{5}{3}}+|V(1,1)-V(0,0)|^2),
	\end{align}
	where in the last inequality, we have used \eqref{equ:zeta4} and
	\begin{align*}
		\int_{\bar M}|\zeta|^2|\Div_{f}^2\zeta|\,\d V_{s,t}&\leq C\int_{\bar M}\big([\zeta']_0^2+(1+|\vec x|^4)\alpha^2\big)\big([\zeta']_2+(1+|\vec x|)[\zeta']_1+(1+|\vec x|^2)[\zeta']_0+(1+|\vec x|^2)\alpha\big)\,\d V_{s,t}\\
		&\leq C(\alpha^3+\mathfrak{X}^{\frac{5}{3}}),
	\end{align*}
	and by Lemma \ref{lem:1stderi} and \eqref{equ:zeta4}, 
	\begin{align*}
		\int_{\bar M}[\zeta]_0^2[\na Q]_1\,\d V_{s,t}&\leq C\int_{\bar M}[\zeta]_0^2 \left([\na Q]_1(0,0)+[\zeta]_2^2+\beta[\zeta]_2 \right) \,\d V_{s,t}\\
		&\leq C\int_{\bar M}[\zeta]_0^2\left( [\zeta']_2+[\na q']_1+\beta^2+[\zeta]_2^2\right)\,\d V_{s,t}\\
		&\leq C\int_{\bar M}[\zeta]_0^4+[\zeta']_2^2+[\na q']_1^2+\beta^2 [\zeta]_0^2+ [\zeta]_0^2[\zeta]_2^2 \,\d V_{s,t}\\
		&\leq C(\alpha^4+\mathfrak{X}^{\frac{5}{3}}+\beta^2(\mathfrak{X}^{\frac{5}{3}}+\alpha^2))+C(\alpha^4+\beta^4+\delta\mathfrak{X}^{\frac{5}{3}})\\
		&\leq C(\alpha^4+\beta^4+\mathfrak{X}^{\frac{5}{3}}).
	\end{align*}
	Here, in the fourth inequality, we need to use the fact $[\zeta]_2\leq C\delta$ and
	\begin{align*}
		\int_{\bar M}[\zeta]_0^2[\zeta]_2^2\,\d V_{s,t}&\leq C\int_{\bar M}[\zeta']_2^2 [\zeta]_0^2+(1+|\vec x|^4)\alpha^2[\zeta]_0^2 \,\d V_{s,t}\\
		&\leq C(\delta^2 \mathfrak{X}^{\frac{5}{3}}+\alpha^2(\alpha^2+\mathfrak{X}^{\frac{5}{3}}))
		\leq C(\alpha^4+\beta^4+\mathfrak{X}^{\frac{5}{3}}).
	\end{align*}
	By Lemma \ref{lem:2ndderi2}, $|w_{ss}|\leq C\big( [\zeta]_2^2+[\na q]_1^2+(1+|\vec x|)([\zeta]_1^2+[\zeta]_0[\na q]_0)+(1+|\vec x|^2)[\zeta]_0^2\big)$, thus it follows from Proposition \ref{pro:Qesti2} that
	\begin{align}\label{3rdWesti7}
		\left|\int_{\bar M}Qw_{ss}\,\d V_{s,t}\right|&\leq C\alpha^2\int_{\bar M}(1+|\vec x|^6)[Q]_0\,\d V_{s,t}+C\delta\int_{\bar M}[\zeta']_2^2+[\na q']_1^2\,\d V_{s,t}\nonumber\\
		&\leq C\alpha^2(\alpha^2+\beta^2+\mathfrak{X}^{\frac{5}{6}}+|V(1,1)-V(0,0)|)+C\delta\mathfrak{X}^{\frac{5}{3}}\nonumber\\
		&\leq C(\alpha^4+\beta^4+\mathfrak{X}^{\frac{5}{3}}+|V(1,1)-V(0,0)|^2),
	\end{align}
	where in the first inequality, we have used the fact that $|Q|\leq C\delta$. By a similar argument, we can also obtain
	\begin{align}
	\label{equ:3rdderi8}\int_{\bar M}|w-1|(|Q_{ss}|+|QQ_s|+|Q^3|)\,\d V_{s,t}\leq C(\alpha^3+\beta^4+\mathfrak{X}^{\frac{5}{3}}+|V(1,1)-V(0,0)|^2).	
	\end{align}
	Combining \eqref{equ:3rdWsss}, \eqref{3rdWesti1}, \eqref{3rdWesti2}, \eqref{3rdWesti3}, \eqref{3rdWesti4}, \eqref{3rdWesti5}, \eqref{3rdWesti6}, \eqref{3rdWesti7} and \eqref{equ:3rdderi8}, we can conclude item (i).
	
	For item (ii), we only give all details of $\sup_{s,t\in [0,1]}|W_{sstt}|\leq C (\alpha^4+\beta^4+\mathfrak{X}^{\frac{5}{3}}+|V(1,1)-V(0,0)|^2)$ and other estimates follow in a similar way. Recall that the formula of $W_{stt}$ is given in \eqref{3rdderiW3}. By similar argument as in proof of Proposition \ref{2ndderiat0} and Lemma \ref{4thderiV}, all terms in $W_{sstt}$ except those involving $\Phi_{st},\Phi_{tt},\Phi_{stt}$ and third derivatives of $w$ can be estimated by $C(\alpha^4+\beta^4+\mathfrak{X}^2)$. The remaining terms in $W_{sstt}$ involving $\Phi_{st},\Phi_{tt},\Phi_{stt}$ are
	$$\int_{\bar M}(\Phi_{st})_{ij}\zeta_{jl}(h_1+th_2)_{il}+(\Phi_{tt})_{ij}\zeta_{jl}\zeta_{il}+\la \Phi_{tt},\zeta\ra Q +\la\Phi_{st},\zeta\ra P+\la\Phi_{stt},\zeta\ra \,\d V_{s,t}.$$
	
	Now we claim
	\begin{align}\label{4thWesti1}
	&\left|\int_{\bar M}(\Phi_{st})_{ij}\zeta_{jl}(h_1+th_2)_{il}+(\Phi_{tt})_{ij}\zeta_{jl}\zeta_{il}+\la \Phi_{tt},\zeta\ra Q +\la\Phi_{st},\zeta\ra P+\la\Phi_{stt},\zeta\ra \,\d V_{s,t} \right|\nonumber\\
	&\leq C (\alpha^4+\beta^4+\mathfrak{X}^{\frac{5}{3}}+|V(1,1)-V(0,0)|^2).
	\end{align}
	Note that by Lemma \ref{3rdphiesti}, we have
	\begin{align}\label{4thWesti2}
		\left|\int_{\bar M}\la \Phi_{stt}, \zeta\ra\,\d V_{s,t}\right|&\leq C\beta^2\int_{\bar M}[\zeta]_0\big([\zeta]_2+(1+|\vec x|)[\zeta]_1+[\na q]_0\big)\,\d V_{s,t}\nonumber\\
		&\leq C\beta^2\int_{\bar M}[\zeta]_2^2+[\na q]_0^2\,\d V_{s,t}\nonumber\\
		&\leq C\beta^2\big(\alpha^2+\mathfrak{X}^{\frac{5}{3}}\big) \leq C(\alpha^4+\beta^4+\beta^2\mathfrak{X}^{\frac{5}{3}}).
	\end{align}
	By Propositions \ref{pro:Pesti} and Lemma \ref{lem:2ndderi3} (iv), we obtain
	\begin{align}\label{4thWesti3}
		\left|\int_{\bar M}(\Phi_{tt})_{ij}\zeta_{jl}\zeta_{il}\,\d V_{s,t}\right|&\leq C\beta^2\int_{\bar M}(1+|\vec x|)[\zeta]_0^2\,\d V_{s,t}\nonumber\\
		&\leq C\beta^2(\alpha^2+\mathfrak{X}^{\frac{5}{3}}) \leq C(\alpha^4+\beta^4+\mathfrak{X}^{\frac{10}{3}}),
	\end{align}
	and
	\begin{align}
	\label{equ:4thWesti44}\left|\int_{\bar M}\la \Phi_{tt},\zeta\ra Q\, \d V_{s,t} \right|&	\leq C\beta^2 \int_{\bar M} (1+|\vec x|)[\zeta]_0[Q]_0\, \d V_{s,t}\nonumber\\
	&\leq C\beta^2 \lc \int_{\bar M} (1+|\vec x|^2)[\zeta]_0^2 \,\d V_{s,t}\rc^{1/2} \lc \int_{\bar M}|Q|^2\,\d V_{s,t}\rc^{1/2}\nonumber\\
	&\leq C\beta^2(\alpha+\mathfrak X^{\frac{5}{6}})(\alpha^2+\beta^2+\mathfrak X^{\frac{5}{6}}+|V(1,1)-V(0,0)|)\nonumber\\
	&\leq C(\alpha^4+\beta^4+\mathfrak X^{\frac{5}{3}}+|V(1,1)-V(0,0)|^2),
	\end{align}
and
	\begin{align}\label{4thWesti4}
		\left|\int_{\bar M}\la\Phi_{st},\zeta\ra P\,\d V_{s,t}\right|&\leq C\beta\int_{\bar M}\big((1+|\vec x|)[\zeta]_1+[\na q]_0\big)[P]_0\,\d V_{s,t}\nonumber\\
		&\leq C\beta\lc\int_{\bar M}[\zeta]_2^2+[\na q]_1^2\,\d V_{s,t}\rc^{1/2}\lc\int_{\bar M}P^2\,\d V_{s,t}\rc^{1/2}\nonumber\\
		&\leq C\beta(\alpha+\mathfrak{X}^{\frac{5}{6}})(\alpha^2+\beta^2+\mathfrak{X}^{\frac{5}{6}})\leq C(\alpha^4+\beta^4+\mathfrak{X}^{\frac{5}{3}}),
	\end{align}
	and
	\begin{align}\label{4thWesti5}
		\left|\int_{\bar M}(\Phi_{st})_{ij}\zeta_{jl}(h_1+th_2)_{il}\,\d V_{s,t}\right|&\leq C\beta^2\int_{\bar M}\big((1+|\vec x|)[\zeta]_1+[\na q]_0\big)[\zeta]_0\,\d V_{s,t}\nonumber\\
		&\leq C\beta^2\lc\int_{\bar M}[\zeta]_2^2+[\na q]_0^2\,\d V_{s,t}\rc^{1/2}\lc\int_{\bar M}[\zeta]_0^2\,\d V_{s,t}\rc^{1/2}\nonumber\\
		&\leq C\beta^2(\alpha+\mathfrak{X}^{\frac{5}{6}})^2 \leq C(\alpha^4+\beta^4+\beta^2\mathfrak{X}^{\frac{5}{3}}).
	\end{align}
	Combining \eqref{4thWesti2}, \eqref{4thWesti3}, \eqref{equ:4thWesti44}, \eqref{4thWesti4} and \eqref{4thWesti5}, we obtain \eqref{4thWesti1}. For those terms involving third derivatives of $w$, we can use Lemma \ref{lem:3rdw} and similar argument as above to obtain the estimates. This completes the proof.
\end{proof}

By Taylor expansion, we have
\begin{align}\label{taylorW}
	&\left|W(1,1)-W(0,0)-W_s-W_t-\frac{1}{2}(W_{ss}+2W_{st}+W_{tt})-\frac{1}{6}(W_{ttt}+3W_{stt}+3 W_{sst})\right|\nonumber\\
	\leq& \frac{1}{6}|W_{sss}(0,0)|+ \frac{2}{3}\sup_{s,t\in [0,1]} |W^{(4)}(s,t)|.
\end{align} 
Here, on the left hand side, all derivatives are evaluated at $t=s=0$ and $W^{(4)}$ denotes fourth derivatives of $W$.

Now we can prove the following theorem.

\begin{thm}\label{lojaforF}
For the weighted Riemannian manifold $(\bar M,g,f)$ such that $\{h,\chi\}=\{g-\bar g, f-\bar f\}$ are compactly supported, if
$$\rVert h\rVert_{C^2}+\rVert \chi\rVert_{C^2}\leq \delta,$$
for a small constant $\delta$ depending only on $n,\, N$, then we have
	\begin{align*}
		|\WW(g,f)-\Theta_{N,m}|\leq C(n,m,N)\left(\rVert u\rVert_{L^2}^3+\rVert h_1\rVert_{L^2}^4+\rVert u\rVert_{L^2}^2\rVert h_1\rVert_{L^2} +\mathfrak{X}^{\frac{5}{3}}+|\VV(g,f)-1|\right),
	\end{align*}
	where $u,\, h_1$ and $\mathfrak{X}$ are defined as in \eqref{equ:defuh_1} and \eqref{equ:defmathfrakX}.
\end{thm}
\begin{proof}
We adopt the notations as in this section and all constants $C$ depend on $n,\,m,\,N$.
	By \eqref{taylorW} and Proposition \ref{4thderiW}, we obtain
	\begin{align}\label{taylorW1}
		&\left|W(1,1)-W(0,0)-W_s-W_t-\frac{1}{2}(W_{ss}+2W_{st}+W_{tt})-\frac{1}{6}(W_{ttt}+3W_{stt}+3 W_{sst})\right|\nonumber\\
		\leq& C(\alpha^3+\beta^4+\alpha^2\beta+\mathfrak{X}^{\frac{5}{3}}+|V(1,1)-V(0,0)|^2).
	\end{align} 
	Here all the derivatives on the left hand side are evaluated at $s=t=0$. Next we claim that
	\begin{align}\label{mainerror1}
		&\left|\frac{1}{2}\int_{\bar M}\la \LL\zeta+2\Div_{\bar f}^*\, \Div_{\bar f}\zeta+2\na^2 Q(0,0),\zeta\ra \,\d V_{0,0}+\int_{\bar M}2|\na Q(0,0)|^2-Q^2(0,0)+Q(0,0)\Div_{\bar f}\,\Div_{\bar f} \zeta\,\d V_{0,0}\right|\nonumber\\
		\leq &C(\alpha^4+\beta^4+\mathfrak{X}^{\frac{5}{3}}+|V(1,1)-V(0,0)|^2).
	\end{align}
	In fact, using the fact that $\Div_{\bar f}(ug_N)=0,\, \rVert\na Q(0,0)\rVert_{L^2}\leq C\mathfrak{X}^{\frac{5}{6}}$ and $\rVert\Div_{\bar f}\zeta'\rVert_{L^2}\leq C\mathfrak{X}^{\frac{5}{6}}$, we have
	\begin{align}\label{mainerror2}
		\left|\int_{\bar M}\la \Div_{\bar f}^*\, \Div_{\bar f}\zeta, \zeta\ra \,\d V_{0,0}\right|&=\left|\int_{\bar M}\la \Div_{\bar f}\zeta, \Div_{\bar f}\zeta\ra \,\d V_{0,0}\right|=\rVert\Div_{\bar f}\zeta'\rVert_{L^2}^2\leq C\mathfrak{X}^{\frac{5}{3}},
	\end{align}
	and
	\begin{align}\label{mainerror3}
		\left|\int_{\bar M}\la \na^2Q(0,0), \zeta\ra \,\d V_{0,0}\right|&=\left|\int_{\bar M}\la \na Q(0,0), \Div_{\bar f}\zeta\ra \,\d V_{0,0}\right|=\left|\int_{\bar M}\la \na Q(0,0), \Div_{\bar f}\zeta'\ra \,\d V_{0,0}\right|\nonumber\\
		&\leq \rVert\na Q(0,0)\rVert_{L^2}\rVert\Div_{\bar f}\zeta'\rVert_{L^2} \leq C\mathfrak{X}^{\frac{5}{3}}.
	\end{align}
	Note that by Proposition \ref{pro:Qesti2}, we have $\rVert Q(0,0)\rVert_{L^2}\leq C(\alpha^2+\beta^2+\mathfrak{X}^{\frac{5}{6}}+|V(1,1)-V(0,0)|)$ and thus
	\begin{align}\label{mainerror4}
		\left|\int_{\bar M}Q(0,0)\Div_{\bar f}\,\Div_{\bar f}\zeta\,\d V_{0,0}\right|\leq \rVert Q(0,0)\rVert_{L^2}\rVert \Div_{\bar f}\,\Div_{\bar f}\zeta'\rVert_{L^2}&\leq C(\alpha^2+\beta^2+\mathfrak{X}^{\frac{5}{6}}+|V(1,1)-V(0,0)|)\mathfrak{X}^{\frac{5}{6}}\nonumber\\
		&\leq C(\alpha^4+\beta^4+\mathfrak{X}^{\frac{5}{3}}+|V(1,1)-V(0,0)|^2).
	\end{align}
	By definition of $\isd$, we know that $\LL (u g_N)=0$, therefore
	\begin{align}\label{mainerror5}
	\left|\int_{\bar M}\la \LL\zeta,\zeta\ra\,\d V_{0,0}\right|&=\left|\int_{\bar M}\la \LL\zeta',\zeta'\ra\,\d V_{0,0}\right|\leq C\mathfrak{X}^{\frac{5}{3}}.
	\end{align}
	Combining \eqref{mainerror2}, \eqref{mainerror3}, \eqref{mainerror4} and \eqref{mainerror5}, we get \eqref{mainerror1}. By Proposition \ref{2ndderiat0}, Corollary \ref{taylorofV6}, Lemma \ref{3rdderiWesti}, \eqref{2ndderiWat01}, \eqref{2ndderiWat02}, \eqref{2ndderiWat03}, \eqref{equ:diff3rd}, \eqref{taylorW1} and \eqref{mainerror1}, we can conclude that
	\begin{align*}
	|W(1,1)-W(0,0)|\leq C(\alpha^3+\beta^4+\alpha^2\beta+\mathfrak{X}^{\frac{5}{3}}+|V(1,1)-V(0,0)|).
	\end{align*}
	 This completes the proof.
\end{proof}

By Corollary \ref{cor:rigidineq}, we have $\rVert u\rVert_{L^2} \leq C\mathfrak{X}^{\frac{1}{2}}$ and $\rVert h_1\rVert_{L^2} \leq C\mathfrak{X}^{\frac{1}{3}}$. Therefore, the next corollary follows immediately from Theorem \ref{lojaforF}:

\begin{cor}\label{cor:lojaforF}
For the weighted Riemannian manifold $(\bar M,g,f)$ such that $\{h,\chi\}=\{g-\bar g, f-\bar f\}$ are compactly supported, if
$$\rVert h\rVert_{C^2}+\rVert \chi\rVert_{C^2}\leq \delta,$$
for a small constant $\delta$ depending only on $n, N$, then we have
	\begin{align*}
		|\WW(g,f)-\WW(\bar g,\bar f)|\leq C(n,m,N)\left(\mathfrak{X}^{\frac{4}{3}}+|\VV(g,f)-1|\right),
	\end{align*}
		where $u,\, h_1$ and $\mathfrak{X}$ are defined as in \eqref{equ:defuh_1} and \eqref{equ:defmathfrakX}.
\end{cor}

\section{Lojasiewicz inequalities and strong uniqueness for generalized cylindrical singularities}\label{sec:loja}
In this section, we will prove Theorems \ref{thm:stong} and \ref{thm:loja}, extending results from \cite{fang2025strong}. The method is very similar to that in \cite{fang2025strong}, so we only sketch the proof. As in Section \ref{sec:variationineq}, we fix a weighted generalized cylinder $(\bar M,\bar g,\bar f)$ from Example \ref{exa:cylinder} which has obstruction of order \(3\) and satisfies the spectral condition
\[
    -\frac l2 \notin \operatorname{spec}\left(\LL_N|_{\mathrm{TT}}\right),
    \qquad \forall l\in \mathbb N^+.
\]
Recall
\begin{align*}
\lc \mathcal C^{n-m}_{N}\rc_{-1}=(\bar M,\bar g,\bar f)=\left(\R^{n-m}\times N^{m}, g_E \times g_{N}, \frac{|\vec{x}|^2}{4}+\frac{m}{2}+\Theta_{N,m} \right).
\end{align*}
Set $\bar b=2\sqrt{\bar f}$\index{$\bar b$}, and fix a base point $\bar p \in \bar M$, which is a minimum point of $\bar f$. We denote by $C(A,B,\ldots)$ the constants depending on $A,B$, etc.

\subsection*{Definitions of radius functions}
Throughout this subsection, the weighted Riemannian manifold $\left(M,g,f\right)$ is always assumed to be normalized (see \eqref{equ:nor}) and satisfies, for some constant $C_V>0$,
\begin{align}\label{normalization22}
	\int_{M\setminus B(p,L)} 1 \,\mathrm{d}V_f \leq C_V e^{-\frac{L^2}{15}}, \quad\forall L>0,
\end{align}
where $p$ is a fixed minimum point of $f$. The following definitions are adapted from \cite[Definition 5.1]{fang2025strong} to the current setting:
\begin{defn}\label{defnradii}
	For $\sigma\in (0,1/10)$ and a weighted Riemannian manifold $(M,g,f)$ with $\Phi=\mathbf{\Phi}(g, f)=\dfrac{g}{2}-\Ric(g)-\nabla^2 f$, we define
	\begin{enumerate}[label=\textnormal{(\Alph{*})}]
		\item \emph{($\rA$-radius)}\index{$\rA$} $\rA$ as the largest number $L$ such that there exists a diffeomorphism $\varphi_A$ from $\{\bar b \le L\} \subset \bar M$ onto a subset of $M$ such that
		\begin{align*}
			\left[\bar g-\varphi_A^* g\right]_5+\left[\bar f-\varphi_A^* f\right]_5 \leq e^{\frac{\bar f}{4}-\frac{L^2}{16}},
		\end{align*}
		\item \emph{($\rBC$-radius)}\index{$\rBC$} $\rBC$ as the largest number $L$ such that there exists a diffeomorphism $\varphi_B$ from $\{\bar b \leq L\} \subset \bar M$ onto a subset of $M$ such that
		\begin{align*}
			\left[\bar g-\varphi_B^* g\right]_0+\left[\bar f-\varphi_B^* f\right]_0\leq e^{-\frac{L^2}{33}},
		\end{align*}
		and
		\begin{equation*}
			\int_{\{\bar b\le L\}}\left|\varphi_B^*\Phi\right|^2 \,\mathrm{d}V_{\bar f}\leq e^{-\frac{L^2}{4-\sigma}}.
		\end{equation*}
Moreover, for all $k \in [1, 10^{10} n\sigma^{-1}]$, the $C^k$-norms of $\bar g-\varphi_B^* g$ and $\bar f-\varphi_B^* f$ are bounded by $1$. 
	\end{enumerate}
Here and for the remainder of the paper, all norms $[\cdot]_l$ are taken with respect to $\bar g$, unless explicitly stated otherwise.
\end{defn}

By the same argument as \cite[Proposition 5.2]{fang2025strong}, we have:
\begin{prop}\label{prop:con}
	With assumption \eqref{normalization22}, there exists a constant $L_1=L_1(n,N,C_V, \sigma)>0$ such that if $\rBC\geq L_1$, then
	\begin{equation*}
		\rA\geq \rBC-2.
	\end{equation*}
\end{prop}
\begin{proof}
	The proposition follows from the same argument as in \cite[Proposition 5.2]{fang2025strong}. The only difference is that instead of the rigidity inequality \cite[Theorem 9.1]{colding2015rigidity}, we need to use \cite[Proposition 4.9, Theorem 4.10]{li2023rigidity} (see also Propositions \ref{pro:rigidityineq1} and \ref{pro:rigidityineq2}).
\end{proof}

Combining with Corollary \ref{cor:lojaforF}, we obtain the next Lojasiewicz inequality within $\rBC$-radius, following the same argument as \cite[Theorem 5.4]{fang2025strong}:

\begin{thm}\label{thm:lojarBC}
There exist constants $L_2= L_2(n,N,Y, \sigma)>0$ and $C=C(n,N,Y)>0$ such that the following holds.
	Let $\XX=\{M^n,(g(t))_{t\in I}\}$ be a closed Ricci flow with entropy bounded below by $-Y$. Assume $x_0^*=(x_0,t_0)\in \XX$ and $[t_0-2r^2,t_0]\subset I$. If the weighted Riemannian manifold $\left(M,r^{-2}g(t_0-r^2),f_{x_0^*}(t_0-r^2)\right)$ satisfies $\rBC\geq L_2 $, then
	\begin{equation*}
		\left|\WW_{x_0^*}(r^2)-\Theta_{N,m}\right|\leq C \exp\lc-\frac{\rBC^2}{6} \rc,
	\end{equation*}
	where $\WW_{x_0^*}$ is the pointed $\WW$-entropy at $x_0^*$; see \emph{\cite[Definition 2.7]{fang2025RFlimit}}.
\end{thm}

\subsection*{Comparison of radius functions}
As in \cite[Section 6]{fang2025strong}, we consider a closed Ricci flow $\XX=\{M^n,(g(t))_{t\in I} \}$ with entropy bounded below by $-Y$ such that $[-10,0]\subset I$. Throughout, we fix a spacetime point $x_0^*=(x_0,0)$, define $\tau=-t$, and set
\begin{align*}
  \begin{dcases}
  & \mathrm{d}\nu_t=\mathrm{d}\nu_{x^*_0;t}=(4\pi \tau)^{-\frac n 2}e^{-f} \,\mathrm{d}V_{g(t)}, \\
		& \Phi= \frac{g}{2}-\tau \left( \Ric+\na^2 f\right),\\
		& F=\tau f.
        \end{dcases}
\end{align*}
Moreover, we define $b=2\sqrt{\left|f(-1)\right|}$.

As before, the model space we consider is the weighted cylinder $\lc \mathcal C^{n-m}_{N}\rc_{-1}$. We set $(\bar M, \bar g(t), \bar f(t))$ to be the induced Ricci flow such that $t=0$ is the singular time coupled with $\bar f(t):=|\vec{x}|^2/4\tau+m/2+\Theta_{N,m}$. We set $\bar F=|\vec{x}|^2/4$, $\bar b=2\sqrt{\bar f(-1)}$, and fix $\bar p$ to be a minimum point of $\bar b$.

\begin{defn}[$\r_{C, \delta}$-radius]
	For the weighted Riemannian manifold $\left(M,g(-1),f(-1)\right)$,
	$\r_{C,\delta}$ is defined as the largest $L$ such that there exists a diffeomorphism $\varphi_C$ from $\{\bar b \le L\}$ of $\bar M$ onto a subset of $M$ such that $f \lc \varphi_C(\bar p), -1 \rc \le n$ and
	\begin{align*}
		\left[\bar g-\varphi_C^* g(-1)\right]_2\leq \delta.
	\end{align*}
\end{defn}

\begin{defn}[Entropy radius]
	The \textbf{entropy radius} $\rE$ is defined as
	\begin{align*}
		\exp\left(-\frac{\mathbf{r}^2_E}{4}\right):=\mathcal W_{x_0^*}(1/2)-\mathcal W_{x_0^*}(2).
	\end{align*}\index{$\rE$}
	Here, we implicitly assume that $\mathcal W_{x_0^*}(1/2)-\mathcal W_{x_0^*}(2) <1$ so that $\rE$ is well-defined.
\end{defn}

Within $\r_{C,\delta}$-radius, we can obtain the estimates for metrics and potential as in \cite[Propositions 6.2, 6.3, Corollary 6.6]{fang2025strong}:
\begin{prop}\label{pro:stabilitymetric}
	For any small $\ep>0$, there exists $\bar \delta=\bar \delta(n,N,\ep)>0$ such that if $\delta \le \bar \delta$ and $\r_{C,\delta} \ge \bar \delta^{-2}$, then the following statements hold.
	\begin{enumerate}[label=\textnormal{(\roman{*})}]
	\item On $\left\{\bar b \le \mathbf{r}_{C,\delta}- \bar \delta^{-1}\right\} \times [-9 , -\ep]$,
		\begin{align*}
		\left[\varphi_C^*g(t)-\bar g(t)\right]_{[\ep^{-1}]} \le \ep,
	\end{align*}
where the norm $[\cdot]_{[\ep^{-1}]}$ is taken with respect to $\bar g(t)$.

\item For $(x,t)\in \left\{\bar b \le \mathbf{r}_{C,\delta}- \bar \delta^{-1}\right\} \times [-2 , -1/2]$, we have
	\begin{align*}
		(1-\ep)\bar F(x)-C(n,N,Y,\ep)\leq F(\varphi_C(x),t)\leq (1+\ep)\bar F(x)+C(n,N,Y,\ep).
	\end{align*}
Moreover, for any $1 \le l \le \ep^{-1}$,
		\begin{align*}
[\varphi_C^*F]_{l} \le C(n,N, Y,\ep) e^{\ep l \bar F}
	\end{align*}	
on $\{\bar b \le \r_{C,\delta}- 2 \bar \delta^{-1}\} \times [-2,-1/2]$.	
	\end{enumerate}
\end{prop}

Next, we define a function $\tilde F$ on $M$ such that
\begin{align*} 
	\square \tilde F=-\frac{n}{2} \quad \text{and} \quad \tilde F=F \quad \text{on} \quad t=-2.
\end{align*}
Then using the same argument as \cite[Lemma 6.8]{fang2025strong}, we have:

\begin{prop} \label{pro:estimateforbetterpotential}
In the same setting as Proposition \ref{pro:stabilitymetric}, we have
	\begin{equation*}
		(1-300\ep)\bar F(x)-C(n,N,Y,\ep)\leq \tilde F(\varphi_C(x),t)\leq (1+300\ep)\bar F(x)+C(n,N,Y,\ep),
	\end{equation*}
for any $(x,t)\in \left\{\bar b\leq \r_{C,\delta}-\bar \delta^{-1}\right\}\times [-2,-1/2]$. Moreover, for any $1 \le l \le \ep^{-1}$, we have on $\{\bar b \le \mathbf{r}_{C,\delta}- 2 \bar \delta^{-1}\} \times [-2, -1/2]$,
	\begin{align*}
		[\varphi_C^*\tilde F]_{l} \le C(n,N, Y,\ep) \exp\lc 10^5 \ep l \bar F \rc.
	\end{align*}
\end{prop}

We now fix a small constant $\sigma\in (0,1/10)$, an integer $l \in \mathbb N$ and define:
	\begin{align*}
\bar \delta_l := \frac{1}{2}\bar \delta \lc n, N, 10^{-100}n^{-1} l^{-1}\sigma \rc,
	\end{align*}\index{$\bar \delta_l$}
for $l \in \mathbb N$, where $\bar \delta$ is the function defined in Proposition \ref{pro:stabilitymetric}. By the same contradiction argument as in \cite[Theorem 6.14]{fang2025strong}, using the compactness theory established in \cite{fang2025RFlimit} and Propositions \ref{pro:stabilitymetric} and \ref{pro:estimateforbetterpotential}, we can bound the $\r_{C,\bar \delta_l}$-radius by $\rA$:
\begin{thm}\label{thm:ext1}
For any $\sigma\in (0, 1/10)$, $D>1$ and $l \in \mathbb N$, there exists a large constant $L'=L'(n,N, Y,  \sigma, l, D)>1$ satisfying the following property.

Let $\varphi_A$ be a diffeomorphism corresponding to $\rA$ in Definition \ref{defnradii}. If $\mathbf{r}_A \in [ L', (1-\sigma) \mathbf{r}_E]$, then there exists another diffeomorphism $\varphi$ from $\{\bar b \le \mathbf{r}_A+D \}$ onto a subset of $M$ such that $\varphi=\varphi_A$ on $\{\bar b \le \rA-2\bar \delta_l^{-1}\}$ and
	\begin{align*}
		\left[\varphi^* g(-1)-\bar g\right]_2\leq \bar \delta_l
	\end{align*}
on	$\{\bar b \le \mathbf{r}_A+D \}$. In particular, we have
	\begin{align*} 
		\mathbf{r}_{C,\bar \delta_l} \ge \mathbf{r}_A+D.
	\end{align*}
\end{thm}

Next, based on Theorem \ref{thm:ext1}, we have the following crucial extension result, using the same argument as \cite[Theorem 6.18]{fang2025strong}.

\begin{thm}\label{thm:ext2}
There exists a large constant $\hat L=\hat L(n, N, Y, \sigma)>1$ such that if $\mathbf{r}_A \in [ \hat L, (1-2\sigma) \mathbf{r}_E]$, then
	\begin{align*}
		\rBC \ge \mathbf{r}_A+\bar D_{100}/10.
	\end{align*}
	Here, both $\rA$ and $\rBC$ denote the radius functions for $(M, g(-1), f(-1))$ (cf. Definition \ref{defnradii}), and $\bar D_{100}=10^4\bar \delta_{100}^{-1}$.
\end{thm}

By using Proposition \ref{prop:con} and Theorem \ref{thm:ext2} iteratively as in \cite[Theorem 6.24]{fang2025strong}, we obtain the following result. Note that by \cite[Theorem 2.13 (i)]{fang2025strong}, the assumption \eqref{normalization22} holds for a constant $C_V=C_V(n,Y)$.

\begin{thm} \label{thm:ra}
	Let $\XX=\{M^n,(g(t))_{t\in I}\}$ be a closed Ricci flow with entropy bounded below by $-Y$. Assume $x_0^*=(x_0,0)\in \XX$ and $[-10,0]\subset I$. For any small $\sigma\in (0,1/10)$, there exists a constant $L=L(n,N,Y,\sigma)$ such that if the weighted Riemannian manifold $\left(M,g(-1),f_{x_0^*}(-1)\right)$ satisfies $\rA\geq L$, then
	\begin{align*}
		\min\{\mathbf{r}_A, \rBC\} \ge (1-3\sigma) \mathbf{r}_E.
	\end{align*}
\end{thm}

Combining Theorem \ref{thm:lojarBC} and Theorem \ref{thm:ra}, we obtain the desired Lojasiewicz inequality, i.e. Theorem \ref{thm:loja} (by a standard limiting argument). Note that here we need to choose $\gamma=2(1-3\sigma)^2/3\in (0,2/3)$.

\begin{thm}[Lojasiewicz inequality] \label{thm:lo}
For any $\gamma \in (0,2/3)$, there exist constants $C=C(n,N,Y)$ and $L_\gamma=L(n,N,Y,\gamma)$ such that the following property holds.

Let $\XX=\{M^n,(g(t))_{t\in I} \}$ be a closed Ricci flow with entropy bounded below by $-Y$. Assume $x_0^*=(x_0,t_0)\in\XX$ and $[t_0-10 r^2,t_0]\subset I$. If the weighted Riemannian manifold $\left(M,r^{-2}g(t_0-r^2),f_{x_0^*}(t_0-r^2)\right)$ satisfies $\rA \geq L_\gamma$, then
	\begin{align*}
		\left|\mathcal W_{x_0^*}(r^2)-\Theta_{N,m}\right| \le C \left( \mathcal W_{x_0^*}(r^2/2)-\mathcal W_{x_0^*}(2 r^2) \right)^{\gamma}.
	\end{align*}
\end{thm}

\begin{cor}\label{cor:quantisummabilityW}
	For any small $0<\ep\leq\ep(n,N,Y)$, $\zeta\in (1/3,1)$ and $\delta\leq\delta(n,N,Y,\ep,\zeta)$, the following holds. Let $\XX=\{M^n,(g(t))_{t\in I}\}$ be a closed Ricci flow with entropy bounded below by $-Y$. Fix $x_0^*=(x_0,t_0)\in \XX$ and constants $s_2>s_1>0$. If
	\begin{align*}
		\left|\WW_{x_0^*}(s_1)-\WW_{x_0^*}(s_2)\right|<\delta,
	\end{align*}
and, for any $s\in [s_1,s_2]$, $x_0^*$ is $(n-m,\delta,\sqrt{s})$-$N$-cylindrical (see Definition \ref{def:almost0}), then
	\begin{align*}
		\sum_{s_1\leq r_j=2^{-j} \leq s_2}\left|\WW_{x_0^*}(r_j)-\WW_{x_0^*}(r_{j-1})\right|^{\zeta}<\ep.
	\end{align*}
\end{cor}
\begin{proof}
For a given $\zeta \in (1/3, 1)$, we choose constants $\theta \in (1/2, 1)$ and $q$ such that 
	\begin{align*} 
		\max \left\{1, \frac{1-\zeta}{\zeta} \right\}<q< \theta^{-1}.
	\end{align*}
Notice that by our assumption on $\zeta$, such choices are possible. Then we choose $\gamma$ in Theorem \ref{thm:lo} by $\gamma=\frac{2(1-3\sigma)^2}{3}=\frac{1}{1+\theta}$. The remaining arguments are the same as in \cite[Corollary 6.28]{fang2025strong}.
	\end{proof}

\subsection*{Strong uniqueness of tangent flows at generalized cylindrical singularities}
Using Theorem \ref{thm:lo}, we can prove the strong uniqueness of tangent flows at generalized cylindrical singularities as \cite[Section 7]{fang2025strong}. Let $\XX=\{M^n, (g(t))_{t \in [-T,0)}\}$ be a closed Ricci flow with entropy bounded below by $-Y$, where $0$ is the first singular time. Suppose $(Z, d_Z, \t)$ is the completion of $\XX$ (see Subsection \ref{subsec:cylsing}). Fix a point $z \in Z_0$ and consider the modified Ricci flow $(M, g^z(s),f^z(s))$ with respect to $z$ (see \eqref{equ:MRFintroI}). Next, we prove Theorem \ref{thm:stong}, which we restate here for the reader's convenience.

\begin{thm}\label{thm:stronguni2intro}
In the setting above, suppose $z \in Z_0$ is a generalized cylindrical singularity with respect to $\bar{\mathcal C}^k_N$. Then for any small $\ep>0$, there exists a large constant $\bar s$ such that for any integer $j \ge \bar s$, there exists a diffeomorphism $\psi_{j}$ from $\Omega^j:=\left\{\bar f \le (1-\ep)\log j\right\} \subset \bar M$ onto a subset of $M$ satisfying the compatibility condition
	\begin{align*}
		\psi_{j+1}=\psi_j \quad \mathrm{on}\quad \Omega^j,
	\end{align*}
and for all $s\geq j$, the following decay estimate holds on $\Omega^j$:
	\begin{align*}
	\left[\psi_{j}^*g^z(s)-\bar g \right]_{[\ep^{-1}]}+\left[\psi_{j}^*f^z(s)-\bar f \right]_{[\ep^{-1}]}\leq C(n,Y,\ep)e^{\frac{\bar f}{2}} s^{-\frac{1}{2}+\ep}.
	\end{align*}
Here, for any integer $l \ge 0$, the norm is defined by
	\begin{align*}
[\cdot]_l:=\sum_{i=0}^l \left|\na_{\bar g}^i (\cdot) \right|_{\bar g}.
\end{align*}
\end{thm}
\begin{proof}
	We only sketch the proof here. First, we set $W(s)=\WW(g^z(s),f^z(s))$, which is increasing in $s$. The next claim gives an integral estimate of entropy gap and follows from the same argument as \cite[Claim 7.3]{fang2025strong}:
	\begin{claim}\label{cla:sumW}
		For any small $\ep>0$, if $s$ is sufficiently large and $\zeta\in (1/3,1)$, then
\begin{align}
    \label{equ:integrability}\int_{s}^\infty \left|W(s'+\log 2)-W(s'-\log 2)\right|^\zeta \,\d s'\leq C_1(n,N,Y,\ep) s^{-3\zeta+1+\ep}.	
    \end{align}
        \end{claim}

    On the other hand, by the same argument as \cite[Theorem 7.1]{fang2025strong}, we know that if $s$ is sufficiently large, 
    \begin{align}\label{equ:radiusMRF}
    \min \{\rA(s),\rBC (s),\rE(s) \} \geq \sqrt{(8-\ep ) \log s}.	
    \end{align}
    And by Theorem \ref{thm:ra} with modified constants, we have
    \begin{align*}
    \min \{\rA(s),\rBC (s) \}	\geq (1-\ep)\rE.
    \end{align*}

   Let $\varphi_s$ be the diffeomorphism corresponding to $\rA(s)$ in Definition \ref{defnradii}. Then by \cite[Corollary 6.17]{fang2025strong} and its proof, for any $l \in \mathbb{N}$, on the set $\{\bar b \leq \sqrt{(4 - \ep/3)\log s}\}$, we have
	\begin{align*}
		\left[\varphi_s^*\lc \frac{g^z(s)}{2}-\na^2f^z(s)-\Ric(g^z(s))\rc\right]_{l}&\leq C(n,N,Y,\ep,l)\exp\lc \frac{\bar f}{2}-\frac{\rA(s)^2}{8}\rc\nonumber\\
		&\leq C(n,N,Y,\ep,l) e^{\frac{\bar f}{2}}|W(s+\log 2)-W(s-\log 2)|^{\frac{(1-\ep)^2}{2}},
	\end{align*}
provided $s$ is sufficiently large.

Next, we set $\Omega_s:=\varphi_s\lc \left\{\bar b\leq  \sqrt{(4 -\ep/3)\log s}\right\}\rc$. Then by \eqref{equ:radiusMRF}, for sufficiently large $s$, we have on $\Omega_s$,
		\begin{align}\label{estishrinkerquanex001}
		\left[ \frac{g^z(s)}{2}-\na^2f^z(s)-\Ric(g^z(s))\right]_{l}&\leq C(n,N,Y,\ep,l) e^{\frac{\bar f}{2}}|W(s+\log 2)-W(s-\log 2)|^{\frac{(1-\ep)^2}{2}},
			\end{align}	
	where the norm is taken with respect to $g^z(s)$. By taking the trace, we have that on $\Omega_s$, 
	\begin{align}\label{estishrinkerquanex0001}
		\left[ \frac{n}{2}-\Delta f^z(s)-\scal (g^z(s))\right]_{l}&\leq C(n,N,Y,\ep,l) e^{\frac{\bar f}{2}}|W(s+\log 2)-W(s-\log 2)|^{\frac{(1-\ep)^2}{2}}.
			\end{align}
	
	By \eqref{estishrinkerquanex001}, \eqref{estishrinkerquanex0001}, Claim \ref{cla:sumW} and the same argument as \cite[Claim 7.4]{fang2025strong}, the following claim is immediate:
\begin{claim}\label{cla:potential}
	If $s_0$ is sufficiently large, then we have
	\begin{align*}
\Omega^{s_0}:=\varphi_{s_0}\lc\left\{\bar b \leq \sqrt{(4-\ep/2)\log s_0}\right\}\rc \subset \Omega_s, \quad \forall s \ge s_0.
	\end{align*}	
	Moreover, for any $x \in \Omega^{s_0}$,
		\begin{align*}
\abs{\exp \lc -\frac{f^z(x, s)}{2} \rc-\exp \lc -\frac{f^z(x, s_0)}{2} \rc} \le C(n, N,Y, \ep) s_0^{-\frac{1}{2}+4\ep}.
		\end{align*}
	\end{claim}   
    
By Claim \ref{cla:potential}, we conclude that $\Omega^{s_0} \subset \Omega_{s}$ for any $s\geq s_0$, so the estimate \eqref{estishrinkerquanex001} holds on $\Omega^{s_0}$. Since the right-hand sides of these estimates are integrable in $s$ by Claim \ref{cla:sumW}, it follows from the evolution equation:
\begin{align*}
	\begin{cases}
		\partial_s g^z(s) = g^z - 2\operatorname{Ric}(g^z) - 2\nabla^2 f^z, \\[4pt]
		\partial_s f^z(s) = \dfrac{n}{2} - \scal_{g^z} - \Delta f^z,
	\end{cases}
\end{align*}
that on $\left\{\bar b \leq \sqrt{(4-\ep/2)\log s_0}\right\}$, $\lc \varphi_{s_0}^*g^z(s),\varphi_{s_0}^*f^z(s)\rc$ converges smoothly to $\lc g_\infty, f_\infty\rc$ as $s \to \infty$. Moreover, by integration, we obtain from \eqref{equ:integrability} that for all $s \ge s_0$,
		\begin{align}\label{equ:decay1}
		&[\varphi_{s_0}^*g^z(s)-g_\infty]_{[\ep^{-1}]}+[\varphi_{s_0}^*f^z(s)-f_\infty]_{[\ep^{-1}]} \notag \\
		\leq & e^{\frac{\bar f}{2}} \int_{s}^\infty \left|W(s'+\log 2)-W(s'-\log 2)\right|^{\frac{(1-\ep)^2}{2}} \,\d s' \leq C(n,N,Y,\ep)e^{\frac{\bar f}{2}} s^{-\frac{1}{2}+4\ep}.
	\end{align}
   
    The remaining argument is the same as in \cite[Claim 7.5, Lemma 7.6, Theorem 7.7]{fang2025strong}, using \eqref{equ:decay1}. This completes the proof by modifying constants.
\end{proof}

Finally, we mention that using similar methods to those in \cite[Section 8]{fang2025strong}, we can extend results in this section to the case of quotients of generalized cylinders. We only state the strong uniqueness theorem for quotients of generalized cylinders but omit the details of proof. Let $(\bar M_\Gamma, \bar g_\Gamma, \bar f_\Gamma)$ be the quotient of $(\bar M, \bar g,\bar f)$, where $\Gamma\leq \mathrm{Iso}((\CC^{n-m}_N)_{-1})$ acts freely on $(\bar M, \bar g,\bar f)$.

\begin{thm}
In the same setting as Theorem \ref{thm:stronguni2intro}, suppose $z \in Z_0$ is a singularity with one tangent flow given by the associated flow induced by $(\bar M_\Gamma, \bar g_\Gamma, \bar f_\Gamma)$. Then for any small $\ep>0$, there exists a large constant $\bar s$ such that for any integer $j \ge \bar s$, there exists a diffeomorphism $\psi_{j}$ from $\Omega^j:=\left\{\bar f_\Gamma \le (1-\ep)\log j\right\} \subset \bar M_\Gamma$ onto a subset of $M$ satisfying the compatibility condition
	\begin{align*}
		\psi_{j+1}=\psi_j \quad \mathrm{on}\quad \Omega^j,
	\end{align*}
and for all $s\geq j$, the following decay estimate holds on $\Omega^j$:
	\begin{align*}
	\left[\psi_{j}^*g^z(s)-\bar g_\Gamma \right]_{[\ep^{-1}]}+\left[\psi_{j}^*f^z(s)-\bar f_\Gamma \right]_{[\ep^{-1}]}\leq C(n,Y,\ep)e^{\frac{\bar f_\Gamma}{2}} s^{-\frac{1}{2}+\ep}.
	\end{align*}
Here, for any integer $l \ge 0$, the norm is defined by
	\begin{align*}
[\cdot]_l:=\sum_{i=0}^l \left|\na_{\bar g_\Gamma}^i (\cdot) \right|_{\bar g_\Gamma}.
\end{align*}
\end{thm}

\section{Rectifiability of generalized cylindrical singular sets}\label{sec:singset}
In this section, we prove the rectifiability of generalized cylindrical singularities in a noncollapsed Ricci flow limit space. Throughout, we fix a noncollapsed Ricci flow limit space $(Z,d_Z,\t)$, which is obtained as a pointed Gromov-Hausdorff limit of a sequence in $\MM(n, Y, T)$ (see \cite[Section 3]{fang2025RFlimit}). As before, we set 
\begin{align*}
\III:=[-0.98 T,0] \quad \text{and} \quad \III^-:=(-0.98 T,0].
	\end{align*}
Moreover, the regular part is given by a Ricci flow spacetime $(\RR, \t, \partial_\t, g^Z)$. For simplicity, we use $B^*(x,r)$ instead of $B_Z^*(x,r)$ for the metric balls in $Z$ with respect to $d_Z$.

We fix an Einstein manifold $(N,g_N)$ with $\Ric(g_N)=g_N/2$, which has obstruction of order \(3\) and satisfies the spectral condition
\[
    -\frac l2 \notin \operatorname{spec}\left(\LL_N|_{\mathrm{TT}}\right),
    \qquad \forall l\in \mathbb N^+.
\]
Consider the standard Ricci flow solution on the cylinder:
\begin{align*}\index{$\mathcal C^k$}
\mathcal C^k_N:=(\bar M,(\bar g(t))_{t<0},(\bar f(t))_{t<0})=\left(\R^{k}\times N^{n-k}, g_E \times |t| g_{N}, \frac{|\vec{x}|^2}{4|t|}+\frac{n-k}{2}+\Theta_{N,n-k} \right).
\end{align*}
With Theorem \ref{thm:lo} and Corollary \ref{cor:quantisummabilityW}, we can obtain the Lojasiewicz inequality for Ricci flow limit spaces and extend \cite[Proposition 5.1]{fang2025singular} to the setting of generalized cylinders by the same argument:
\begin{thm}
For any $\gamma \in (0,2/3)$, there exist constants $C=C(n,N,Y)$ and $\delta=\delta(n,N,Y,\gamma)$ such that the following property holds.

If $z\in Z_{\III^{-}}$ is $(k,\delta, r)$-$N$-cylindrical, then
	\begin{align*}
		\left|\widetilde \WW_z(r^2)-\Theta_{N,n-k}\right| \le C \left( \widetilde\WW_{z}(r^2/2)-\widetilde\WW_{z}(2 r^2) \right)^{\gamma}.
	\end{align*}
Here, $\widetilde \WW_z$ is the modified pointed $\WW$-entropy on Ricci flow limit spaces, see \emph{\cite[Section 4.1]{fang2025singular}}.
\end{thm}

\begin{prop}\label{sumWonRFlimit}
	For any $0<\ep\leq\ep(n,N,Y)$, $\zeta\in (1/3,1)$, and $\delta\leq\delta(n,N,Y,\ep,\zeta)$, the following holds. Suppose
	\begin{align*}
		\left|\widetilde \WW_{z}(s_1)-\widetilde \WW_{z}(s_2)\right|<\delta,
	\end{align*}
for $0<s_1<s_2$, and for any $s\in [s_1,s_2]$, $z$ is $(k,\delta,\sqrt{s})$-$N$-cylindrical (see Definition \ref{def:almost0}). Then
	\begin{align*}
		\sum_{s_1\leq r_j=2^{-j} \leq s_2}\left|\widetilde \WW_{z}(r_j)-\widetilde \WW_{z}(r_{j-1})\right|^{\zeta}<\ep.
	\end{align*}
\end{prop}

We then introduce the following notion of an $N$-cylindrical neck region, adapted from \cite[Definition 5.9]{fang2025singular}.

\begin{defn}[Cylindrical neck region]\label{defiofcylneckregion}
Given constants $\delta>0$, $\cc \in (0, 10^{-10 n })$, $r>0$ and $z \in Z_{\III^-}$ with $\t(z)-2\delta^{-1}r^2 \in \III^-$, we call a subset $\NNN \subset B^*(z, 2r)$ a \textbf{$(k,\delta, \cc, r)$-$N$-cylindrical neck region}\index{$(k,\delta,\cc,r)$-cylindrical neck region} if $\NNN=B^*(z,2r)\setminus B_{r_x}^*(\CCC)$, where $\CCC \subset B^*(z, 2r)$ is a nonempty closed subset with $r_x: \CCC \to \R_{+}:=[0, +\infty)$, satisfies\textup{:}
	\begin{itemize}[leftmargin=*, label={}]
		\item \emph{(n1)} for any $x, y \in \CCC$, $d_Z(x, y) \ge \cc^2(r_x+r_y);$\index{$\cc$}
		\item \emph{(n2)} for all $x\in \CCC$,
		$$\widetilde \WW_{x}(\delta r^2_x)-\widetilde \WW_{x}(\delta^{-1} r^2)<\delta;$$
		\item \emph{(n3)} for each $x\in \CCC$ and $\cc^2 r_x\leq s\leq 2r$, $x$ is $(k,\delta,s)$-$N$-cylindrical with respect to $\mathcal L_{x,s};$
		\item \emph{(n4)} for each $x\in \CCC$ and $\cc^{-5} r_x\leq s\leq r-d_Z(x, z)/2$, we have $\mathcal{L}_{x,s} \cap B^*(x, s) \subset B^*_{\cc s}(\CCC)$ and $\CCC\bigcap B^*(x,s)\subset B^*_{\cc s}(\mathcal{L}_{x,s})$.
	\end{itemize}
Here, $\CCC$\index{$\CCC$} is called the \textbf{center} of the neck region, and $r_x$\index{$r_x$} is referred to as the \textbf{radius function}. We decompose $\CCC=\CCC_0\bigcup\CCC_{+}$\index{$\CCC_0$}\index{$\CCC_+$}, where $r_x>0$ on $\CCC_+$ and $r_x=0$ on $\CCC_0$. In addition, we use the notation
	\begin{align*}
	B^*_{r_x}(\CCC):=\CCC_0 \bigcup \bigcup_{x \in \CCC_+} B^*(x, r_x).
	\end{align*}
\end{defn}

We denote by \(S^k(N)\subset S^k\) the set consisting of all points at which some tangent flow is given by the generalized cylinder $\bar \CC_N^k$.
By the same decomposition argument as \cite[Proposition 5.11, Theorem 5.21]{fang2025singular}, we can obtain the following construction of an $N$-cylindrical neck region and its structure result.
\begin{prop}\label{pro:cylneckdecomp}
Given $\delta>0$ and $\cc \in (0, 10^{-10 n})$, for any $x_0\in \MS^k(N)$ with some tangent flow given by $\bar \CC_N^k$, there exist $r_0>0$ and a subset $\CCC=\CCC_0\bigcup\CCC_+\subset B^*(x_0,2r_0)$ with $r_x:\CCC\to \R_{+}$ satisfying $r_x>0$ on $\CCC_+$ and $r_x=0$ on $\CCC_0$, such that the following hold\textup{:}
	\begin{enumerate}[label=\textnormal{(\roman{*})}]
		\item $\NNN=B^*(x_0,2r_0)\setminus B^*_{r_x}(\CCC)$ is a $(k,\delta,\cc, r_0)$-$N$-cylindrical neck region.
		\item $\MS^k(N) \bigcap B^*(x_0,r_0)\subset\CCC_0$.
	\end{enumerate}
\end{prop}

\begin{prop}[Ahlfors regularity]\label{pro:ahlforsregucyl1}
Given $\ep>0$, if $\cc \le \cc(n,N)$ and $\delta \le \delta(n, N, Y, \cc,\ep)$, then for the $(k,\delta, \cc,r_0)$-$N$-cylindrical neck region $\NNN=B^*(x_0 ,2r_0)\setminus B_{r_x}^*(\CCC)$ constructed in Proposition \ref{pro:cylneckdecomp}, any $x\in\CCC$ and $r_x\leq s\leq r_0-d_Z(x, x_0)/2$, we have
	\begin{align*} 
		 C^{-1}(n, N, Y, \cc) s^k\leq\mu(B^*(x,s))\leq C(n, N, Y, \cc)s^k.
	\end{align*} 
Moreover, we can find a countable collection of $\HHH^k$-measurable subsets $ E_i\subset \CCC_0$ such that $\HHH^k(\CCC_0\setminus\bigcup_i E_i)=0$ and for each $i$, there exists a bi-Lipschitz map $u_i:E_i\to \R^k$, where $\R^k$ is equipped with the standard Euclidean distance, such that
	\begin{align*}
\sqrt{|\t(x)-\t(y)|}\leq\ep d_Z(x,y),\quad \forall x, y \in E_i.
	\end{align*} 
	\end{prop}

Combining Propositions \ref{pro:cylneckdecomp}, \ref{pro:ahlforsregucyl1} with a standard covering argument, the next theorem is immediate (see also \cite[Theorem 1.3]{fang2025singular}). For the notion of horizontal parabolic rectifiability, see \cite[Definition 1.2]{fang2025singular}. 
\begin{thm} 
The set $\MS^k(N)$ is horizontally parabolic $k$-rectifiable with respect to $d_Z$.
\end{thm}

We note that the above theorem can also be generalized to quotients of generalized cylinders using the same argument:
\begin{thm}
The set $\MS^k_{\mathrm{qc}}(N)$, consisting of all points at which some tangent flow is given by the generalized cylinder $\bar\CC_N^k$ or its quotient, is horizontally parabolic $k$-rectifiable with respect to $d_Z$.
\end{thm}

\bibliographystyle{alpha}
\bibliography{lojageneralcyl}
\vskip10pt

Hanbing Fang, Mathematics Department, Stony Brook University, Stony Brook, NY 11794, United States; Email: hanbing.fang@stonybrook.edu;\\

Yu Li, Institute of Geometry and Physics, University of Science and Technology of China, No. 96 Jinzhai Road, Hefei, Anhui Province, 230026, China; Hefei National Laboratory, No. 5099 West Wangjiang Road, Hefei, Anhui Province, 230088, China; E-mail: yuli21@ustc.edu.cn. \\

\end{document}